\documentclass[11pt]{amsart} 
\usepackage{amsmath}
\usepackage{amsfonts}
\usepackage{amssymb}
\usepackage{mathrsfs} 
\usepackage{amsthm}
\usepackage{palatino}
\usepackage[margin=2.5cm, vmargin={1.5cm}]{geometry}
\usepackage[pdfauthor={Cormac O'Sullivan and Morten S. Risager},
pdftitle={Non-vanishing of Taylor coefficients and Poincar\'e series},
            pdfsubject={Primary 11F71; Secondary  11M36, 34E10},
            pdfkeywords={Write-abstract.},
            pdfproducer={LaTeX2e with hyperref},
            pdfcreator={Pdflatex},
]{hyperref}

\usepackage{pst-node}
\usepackage{pst-coil}
\usepackage{pst-plot}
\usepackage{pst-grad}
\usepackage{pstricks}
\newif\ifpdf
\ifx\pdfoutput\undefined
  \pdffalse
\else
  \ifnum\pdfoutput=1
    \pdftrue
  \else
    \pdffalse
  \fi
\fi
\ifpdf
\usepackage{pst-pdf} 
\fi

\newrgbcolor{LemonChiffon}{1.0 0.98 0.8}
\newrgbcolor{magenta}{1.0 0.3 0.4}
\newrgbcolor{mistyrose}{1.0 0.0 0.6}
\newrgbcolor{deeppink}{1.0 0.08 0.58}
\newrgbcolor{lightsalmon}{1.0 0.63 0.48}
\newrgbcolor{lightred}{0.7 0.0 0.0}
\newrgbcolor{lightblue}{0.0 0.0 0.5}
\newrgbcolor{lightblu2}{0 0.4 1}
\newrgbcolor{indianred}{0.80  0.36  0.36}
\newrgbcolor{lightgreen}{0.56 0.93 0.56}
\newrgbcolor{byellow}{0.8 0.95 1}

\title{Non-vanishing of Taylor coefficients and Poincar\'e series}
\author{Cormac O'Sullivan}
\address{Department of Mathematics\\
 The CUNY Graduate Center \\ New York, NY 10016-4309\\ U.S.A.}
\email{cosullivan@gc.cuny.edu}

\author{Morten S. Risager}
\address{Department of Mathematical Sciences\\
University of Copenhagen\\
Universitetsparken 5, 2100 Copenhagen $\emptyset$ \\ Denmark.} \email{risager@math.ku.dk}

\date{Jan 4, 2012}
\thanks{The second author was supported by a grant from The Danish Natural Science Research Council.}
\begin{document}

\begin{abstract}
We prove recursive formulas for the  Taylor coefficients of
cusp forms, such as  Ramanujan's Delta function, at points in the upper half-plane. This allows us to show the non-vanishing of all Taylor coefficients of  Delta at CM points of small discriminant as well as the non-vanishing of certain Poincar\'e series. At a
``generic'' point all Taylor coefficients are shown to be non-zero. Some conjectures on the Taylor coefficients of Delta at CM points are stated.
\end{abstract}

\maketitle

\def\s#1#2{\langle \,#1 , #2 \,\rangle}

\def\H{{\mathbb  H}}
\def\O{{\mathcal O}}
\def\w{{r  }}
\def\F{{\mathfrak F}}
\def\C{{\mathbb C}}
\def\R{{\mathbb R}}
\def\Z{{\mathbb Z}}
\def\Q{{\mathbb Q}}
\def\N{{\mathbb N}}
\def\st{{\mathbb S}}
\def\D{{\mathbb D}}
\def\fun{{\mathbb F}}
\def\B{{\mathbb B}}
\def\G{{\Gamma}}
\def\GH{{\G \backslash \H}}
\def\g{{\gamma}}
\def\L{{\Lambda}}
\def\ee{{\varepsilon}}
\def\K{{\mathcal K}}
\def\Re{\text{\rm Re}}
\def\Im{\text{\rm Im}}
\def\SL{\text{\rm SL}}
\def\GL{\text{\rm GL}}
\def\PSL{\text{\rm PSL}}
\def\sgn{\text{\rm sgn}}
\def\tr{\text{\rm tr}}
\def\lqs{\leqslant}
\def\gqs{\geqslant}
\def\z{\mathfrak{z}}
\def\dd{\mathcal{D}}

\def\ca{{\frak a}}
\def\cb{{\frak b}}
\def\cc{{\frak c}}
\def\cd{{\frak d}}
\def\ci{{\infty}}

\def\sa{{\sigma_\frak a}}
\def\sb{{\sigma_\frak b}}
\def\sc{{\sigma_\frak c}}
\def\sd{{\sigma_\frak d}}
\def\si{{\sigma_\infty}}

\def\se{{\sigma_\eta}}
\def\sz{{\sigma_{z_0}}}
\def\t6{{\tau}}
\def\e6{{\mathcal E}}


\newtheorem{theorem}{Theorem}[section]
\newtheorem{lemma}[theorem]{Lemma}
\newtheorem{prop}[theorem]{Proposition}
\newtheorem{cor}[theorem]{Corollary}
\newtheorem{conj}[theorem]{Conjecture}
\newtheorem{example}[theorem]{Example}
\newtheorem{remark}{Remark}

\newcounter{coundef}
\newtheorem{adef}[coundef]{Definition}

\newcounter{thm1count}
\newtheorem{thm1}[thm1count]{Theorem}

\newcommand{\x}{\scriptstyle}

\renewcommand{\labelenumi}{(\roman{enumi})}

\numberwithin{equation}{section}


\newcommand{\abs}[1]{\left\lvert #1 \right\rvert}
\newcommand{\norm}[1]{\left\lVert #1 \right\rVert}

\bibliographystyle{plain}

\section{Introduction}

\subsection{Background} Let $\G=\SL(2,\Z)$ be the full modular group acting on the upper half plane $\H$.  Then $\G$ has a cusp at $\infty$
with stabilizer
$\Gamma_\infty=\{\pm(\smallmatrix 1 & n \\ 0 & 1 \endsmallmatrix)|n \in \Z\}$. Denote the space of holomorphic cusp forms for $\G$ of even weight $k$ by $S_k(\G)$. Every $f \in S_k(\G)$ has a $q$-expansion of the form
\begin{equation}\label{af}
    f(z)=\sum_{m=1}^\infty c_\ci(f,m) q^m,
\end{equation}
for $q=e^{2\pi i z}$. The famous inhabitant of the one-dimensional space $S_{12}(\G)$ is
\begin{equation}\label{dl0}
\Delta(z)  :=  q \prod_{n=1}^\infty (1-q^{n})^{24}  =  q -24q^2+ 252q^3 -1472q^4+ 4830q^5  + O(q^6),
\end{equation}
Ramanujan's Delta function. Ramanujan discovered many of the remarkable
arithmetic properties of the  coefficients $\tau(m):=c_\ci(\Delta,m)$ that bear
his name. These properties were later proved, most notably by Ramanujan himself, Mordell
and Deligne. In 1947 Lehmer \cite{Lehmer:1947a} asked if
the following statement is true:
\begin{equation}
\label{lehmersconjecture}\tau{(m)}\neq 0 \quad \text{ for every } \quad m \in \Z_{\gqs 1}.
\end{equation}
We will refer to (\ref{lehmersconjecture}) as \emph{Lehmer's
  conjecture}.  Lehmer initially verified that $\tau(m)\neq 0$ for $m<3316799$. This has
been greatly improved over the years and is now verified for $m <
2 \cdot 10^{19}$
\cite{Bosman:2007a}.  Recently   Lehmer's conjecture has found an interesting interpretation in terms of  spherical $t$-designs: it is equivalent to the shell of norm $2m$ of the E8 lattice never being a spherical 8-design for $m \in \Z_{\gqs 1}$. See for example    \cite{BannaiMiezaki:2010} and the contained references for this connection.

Associated with each point $z_0=\alpha+i \beta \in \H$ there is another natural   expansion for $\Delta$, less well-known than \eqref{dl0}, that we describe next.  Let $\D$ be the open unit
disc,  centered at the origin in $\C$, and set
\begin{equation}\label{scaling}
\sz =  \frac{1}{2i \beta}\begin{pmatrix} -\overline{z_0} & z_0 \\ -1 & 1 \end{pmatrix} \in \GL(2,\C).
\end{equation}
Then $\sz 0 = z_0$, $\sz \infty = \overline{z_0}$ and  $z \mapsto \sz z$ is a biholomorphic map  $\D \to \H$. Let $f:\H \to \C$.
The stroke operator $|$ is
defined as usual by
$$
 \left( f|_k \g \right)(z) := \frac{\det(\g)^{k/2} f(\g z)}{j(\g,z)^k} \quad \text{ for} \quad\g=(\smallmatrix a & b \\ c & d \endsmallmatrix),\quad j(\g,z):=cz+d.
$$
If $f$ is holomorphic in $\H$ then
$f\vert_k\sigma_{z_0}(z)$ is holomorphic in $\D$ with a Taylor
expansion at zero:
\begin{equation} \label{czfn}
f\vert_k\sigma_{z_0}(z)=\sum_{n=0}^\infty c_{z_0}(f,n)z^n.
\end{equation}

\begin{example}
{\em
At $z_0=i$ and $z_0=\omega:=e^{2\pi i/3}$, the elliptic fixed points for $\G$, we have
\begin{eqnarray}\label{etau1}
\frac{\bigl( \Delta|_{12} \sigma_i \bigr) (z)}{-64\Delta(i)}
& = &    1-12\frac{(\w_iz)^2}{2!}+ 216\frac{(\w_iz)^4}{4!}+10368\frac{(\w_iz)^6}{6!}     + O(z^{8}) \\
\frac{\bigl( \Delta|_{12} \sigma_\omega \bigr) (z)}{ -27\Delta(\omega)}
& = &
1+48\frac{(\w_\omega z)^3}{3!}+
   18432\frac{(\w_\omega z)^6}{6!}  +13271040 \frac{(\w_\omega z)^9}{9!} + O(z^{12}) \label{etau2}
\end{eqnarray}
where
$r_i=-\Gamma(1/4)^{4}/(8\sqrt{3}\pi^{2})$ and
$r_\omega=-\sqrt{3}\Gamma(1/3)^{6}/(16\pi^{3})$. See \S \S \ref{i} and \ref{w} for \eqref{etau1} and \eqref{etau2}.
}
\end{example}

The Taylor expansion \eqref{czfn} is a very natural one to use as its radius of convergence  is $1$ and
$$
f(z)=\frac{(z_0-\overline{z_0})^{k/2}}{(z-\overline{z_0})^{k}}\sum_{n=0}^{\ci} c_{z_0}(f,n)\left( \sigma_{z_0}^{-1} z\right)^n
$$
is valid for all $z \in \H$. The numbers $c_{z_0}(f,n)$ in \eqref{czfn} have a dual nature. Certainly they are the (normalized) Taylor coefficients of $f$ at $z_0$, and as Taylor coefficients they may be found from the derivatives of $f$, as in Proposition \ref{bim}. This property will be crucial in \S\S \ref{non-vanishing},  \ref{cma}. It is also important to regard the numbers $c_{z_0}(f,n)$  as Fourier coefficients and obtain them through integration, as in \S \ref{basics-fourier-coefficients}. This also highlights their similarities to the Fourier coefficients at cusps, such as $c_\ci(f,m)$ in \eqref{af}. For these reasons, throughout this paper we call $c_{z_0}(f,n)$ the {\em $n$-th Fourier coefficient of $f$ at $z_0$}.

These non-cuspidal Fourier coefficients have previously been studied  and/or used
  by several authors including
  Petersson \cite{Petersson:1941a}, Mori \cite{Mori:1995,Mori:2011},
  Datskovsky and Guerzhoy \cite{DatskovskyGuerzhoy:2008}, Rodriguez Villegas and
  Zagier \cite{VillegasZagier:1993a}, Imamo{\=g}lu  and
  O'Sullivan \cite{ImamogluOSullivan:2009}, and for non-holomorphic
  analogues by Sarnak \cite{Sarnak:1994a},
  Reznikov \cite{Reznikov:2007a}, and others.

\begin{remark}\label{trivial-vanishing}  {\em We notice that $c_{z_0}(\Delta,n)=0$ unless $n$
is divisible by 2 in \eqref{etau1}, and divisible by 3 in
\eqref{etau2}. This vanishing  is a basic general phenomenon: If $f$ is a form of weight
$k$, and  $z_0$ is a  fixed point of
order $N\gqs 1$ for $\G$ then $c_{z_0}(f,n)=0$ whenever $n\not \equiv -k/2 \bmod  N$. See
\S \ref{basics-fourier-coefficients} for an explanation.
When $n\not \equiv -k/2 \bmod N$ we call the corresponding Fourier coefficients
\emph{trivial}. Note that if $z_0$ is \emph{not} an elliptic fixed point
(i.e. if $N=1$) then there are no trivial Fourier coefficients.}
\end{remark}

\subsection{The main results}
We first investigate Fourier expansions at CM points.
 Let $D$ be a negative fundamental discriminant, i.e. $D<0$
with $D\equiv 1 \bmod 4$ and squarefree, or with $D=4 m$ where $m\equiv
2,3\bmod 4$ and squarefree.  We consider CM points of the form
\begin{equation}\label{zzd}
    \z_{D}:= \begin{cases} \sqrt{D}/2 & \text{for $D$ even} \\ (1+\sqrt{D})/2 & \text{for $D$ odd}, \end{cases}
\end{equation}
such that $D$ is
the discriminant of $\Q(\sqrt{D})$, with ring of integers $\mathcal{O}_{\Q(\sqrt{D})}=\hbox{span}_\Z(1,\z_D)$.

For $\z$ a CM point and $f \in S_k(\G)$ with algebraic Fourier coefficients at $\ci$, it is known (due to the work of  Damerell and others - see for example
\cite{DatskovskyGuerzhoy:2008} for references) that there exist  non-zero complex numbers $\kappa$ and $\lambda$
such that  the quotients
$$
c_{\z}(f,n)/(\kappa \lambda^{n})
$$
are algebraic for all $n$. We show, using a technique of  Rodriguez Villegas and Zagier,  that for $\Delta(z)$ these algebraic numbers are essentially the constant terms
in a recursively defined sequence of polynomials with algebraic
integer coefficients:

\begin{theorem} \label{mainthm}
For every CM point $\z = \z_D$, with $D$ a negative fundamental discriminant, there exists a number field $K$, constants $\kappa_{\z},\lambda_{\z} \in \C$ and $a_i(t)\in \O_K[t]$ (all depending on $\z$ and explicitly given), so that the $m$-th Fourier coefficient of $\Delta(z)$ at $\z$ is
\begin{equation}\label{ty}
c_\z(\Delta,m) = \kappa_{\z}\frac{\lambda_{\z}^m}{m!}q_{m, \z}(0) \qquad(m \gqs 0)
\end{equation}
with  $q_{m, \z}(t) \in \O_K[t]$ defined recursively by
\begin{equation}\label{simplified-writing}
\begin{split}
q_{0, \z}(t)={}& 1, \quad q_{1, \z}(t)=a_{1}(t), \\
q_{n+1,\z}(t)={}&(a_{1}(t)+n \,a_{2}(t))q_{n,\z}(t)+a_{3}(t)q'_{n,\z}(t)+n(n+11)a_{4}(t)q_{n-1,\z}(t) \qquad(n \gqs 1).
\end{split}
\end{equation}
\end{theorem}

Theorem \ref{mainthm} is proved in \S \ref{cma} where $K$, $\kappa_{\z}$, $\lambda_{\z}$ and $a_i(t)$ for $i=1,2,3,4$ are precisely defined.
A set $\mathscr D$ of  negative fundamental discriminants we will focus on is
$$
\mathscr D :=\{-3,-4,-7,-8,-11,-15,-19,-20,-24\}.
$$
Studying the polynomials $q_{m, \z}(t)$ in \eqref{simplified-writing} modulo  $l$ we find
that, for certain primes $l$, $q_{m, \z}(0)$ becomes periodic mod $l$ and all
values in a period are non-zero.  This in turn allows us to conclude that all
the Fourier coefficients $c_{\z}(\Delta,n)$ are non-zero:

\begin{theorem} \label{ell lehmer}
 For $D \in \mathscr D$, all non-trivial Fourier coefficients of $\Delta$ at $\z_D$ are  non-zero.
\end{theorem}


\begin{remark}\label{do they all vanish}
{\em The nine  points we consider in Theorem
\ref{ell lehmer} are especially simple, but it is relatively straightforward to test
other CM points.  We prove in Theorem \ref{periodic} that for every CM point
$\z$ and every prime $l$ the sequence $q_{n, \z}(0), n \gqs 0$ becomes periodic
mod $l$. To prove non-vanishing  we require a
prime $l$  such that the whole corresponding period is non-zero mod
$l$. It is tempting to speculate that such a prime always exists and
therefore that  the non-trivial Fourier
coefficients of $\Delta$ at \emph{all} CM points are always non-zero.}
\end{remark}

The sequences  $q_{n, \z}(0)$ for $n \gqs 0$ possess many interesting and, as yet, unexplained features. For example, let $q_n$ be the  coefficient of $(r_i z)^n/n!$ in \eqref{etau1} i.e. $q_0=1, q_1=0, q_2=-12$, etc. As a special case of more general results in  \S \ref{arit}  we show the following.
\begin{theorem}
For all primes $l$ with $2<l<100$ we have
$$
q_{n} \rightarrow 0 \bmod l \iff l \equiv 3 \bmod 4.
$$
\end{theorem}

We next examine the vanishing of $c_{z_0}(f,n)$ for $f\in S_k(\G)$ and $z_0$ a general point in $\H$. Recall from Remark \ref{trivial-vanishing} that for $z_0 \in \H$ with order $N$,  $c_{z_0}(f,n)$  is non-trivial  when $n \equiv -k/2 \bmod N$. (If $z_0$ is a cusp then we say all $c_{z_0}(f,n)$ for $n\gqs 1$ are non-trivial.)  To highlight the points we are interested in,  define
$$
A_f:=\bigl\{z_0 \in \H \cup \{ \text{cusps of }\G\}\ \big| \ c_{z_0}(f,n) = 0 \text{ \ for some non-trivial Fourier coefficient} \bigr\}.
$$
We would like to know the detailed structure of this set.
With this notation, we see for instance that Lehmer's conjecture may be written succinctly as $\ci \not\in A_\Delta$.

\begin{theorem} \label{almost all non-vanishing}
 For $f \in S_k(\G)$, not identically $0$, the set $A_f$ has measure $0$.
\end{theorem}

Theorem \ref{almost all non-vanishing} asserts that for a generic
point all Fourier coefficients are non-vanishing. Since the countable set of cusps has measure $0$, it is really a statement about non-vanishing at points in $\H$. The proof we give  in \S \ref{basics-fourier-coefficients} works for general Fuchsian groups $\G$ that may not have cusps.

\begin{theorem} \label{0}
 The set $A_\Delta$ is non-empty. It contains  $z_0 \approx 1.344 i$ for which $c_{z_0}(\Delta,2)=0$.
\end{theorem}
Theorem \ref{0} is proved in \S \ref{numer}, where we numerically find further elements of $A_\Delta$.
Lastly, in \S \ref{more on poincareseries}, we generalize an identity of Petersson, \eqref{petform}, and exhibit formulas for the averages
\begin{equation*}
\sum_{f\in \mathcal{F}} c_{z_0}(f,m) \overline{c_\infty(f,n)} \quad
\text{ and } \quad \sum_{f\in \mathcal{F}}
c_{z_0}(f,m) \overline{c_{z_0'}(f,n)}
\end{equation*}
 with $z_0$, $z_0' \in \H$ and $\mathcal{F}$ an orthonormal basis  of $S_k(\G)$.

\vskip 3mm
{\bf Acknowledgements.} We thank Gautam Chinta and J\"urg Kramer for helpful discussions.

\section{Applications}
In this section we describe two applications of our non-vanishing result, Theorem \ref{ell lehmer}.
\subsection{Non-vanishing of Poincar\'e series}
The problem of proving non-vanishing for
Poincar\'e series is very interesting and has been investigated by numerous authors
e.g. \cite{Petersson:1944a,
  Gandhi:1961a,Serre:1985a,Gaigalas:1984a,Metzger:1980a,Serre:1981a,Lehner:1980a,Mozzochi:1989a}.

 Let $e_m(z):=e^{2\pi i mz}$. For $m \in \Z_{\geqslant 1}$, the $m$-th Poincar\'e series of weight
 $4 \leqslant k \in 2 \Z$ associated to the cusp at infinity is
\begin{equation}\label{poincare}
P_{m}(z):= \sum_{\g \in \G_\infty \backslash \G} \left( e_m|_k \g \right)(z) = \sum_{\g \in \G_\infty \backslash \G} \frac{e^{2 \pi i m \g z}}{j(\g,z)^k}.
\end{equation}
 These {\em parabolic} Poincar\'e series span the finite dimensional space $S_k(\G)$.
Petersson \cite{Petersson:1949a} proved that, in fact, $P_1, P_2, \cdots , P_{d(k)}$ form a basis for $S_k(\G)$ with $d(k):=\dim(S_k(\G))$. When $d(k)=0$ (which happens precisely when
$k\in\{4,6,8,10,14\}$)  the Poincar\'e series $P_m(z)$ must vanish
identically. In the simplest non-trivial case, with $k=12$ and $d(k)=1$,
the non-vanishing of every $P_m(z)$ is equivalent to Lehmer's conjecture since
\begin{equation}\label{ptau}
P_m=  \tau(m) \left[\frac{10! }{(4\pi m)^{11}}\right]\frac{\Delta}{\norm{\Delta}^2}
\end{equation}
(See e.g. \cite[(3.29)]{Iwaniec:1997a}). Equation \eqref{ptau} follows from the fact that $S_{12}(\G)$ is one-dimensional
containing $\Delta$ and
the next  result, first demonstrated by Petersson in
\cite{Petersson:1941a}. See also \cite[Theorem 3.3]{Iwaniec:1997a}  for example, for its
proof. Here $\norm{\cdot}$ is the norm corresponding to the Petersson inner product $\s{\cdot}{\cdot}$.
\begin{prop}  \label{porth}
For $f \in S_k(\G)$  and $n \geqslant 1$ we have
\begin{equation*}
\s{f}{P_{n}} =    c_{\ci}(f,n)\left[\frac{ (k-2)!}{(4\pi n)^{k-1}}\right].
\end{equation*}
\end{prop}
We also note that if $m \ll k^{2-\ee}$ then, using
Petersson's formula (\ref{petform}), Rankin \cite[Theorem 1]{Rankin:1980a}  proved that $P_m \not\equiv 0$.

\vskip 3mm
 A second family of Poincar\'e series was introduced by Petersson in \cite{Petersson:1941a} as {\em Poincar\'e series of elliptic type}. They are associated not with cusps but with points $z_0$ in $\H$.
Recall $\sz$  from \eqref{scaling} with $\det(\sz) = 1/(2 i \beta)$ and  $\sz^{-1} =(\smallmatrix 1 & -z_0 \\ 1 & -\overline{z_0} \endsmallmatrix)$ mapping the upper half-plane $\H$  biholomorphically to the open unit
disc $\D$.

Let $\mu_m(z)=z^m$. For $4 \leqslant k \in 2 \Z$ and $m \in \Z_{\geqslant 0}$ define
\begin{eqnarray}
    P_{z_0,m}(z)
      & := &  \sum_{\g \in  \G}
    \bigl(\mu_m |_k \sz^{-1}\g \bigr)(z) \nonumber \\
     & = &  (2i \beta)^{k/2}\sum_{\g \in   \G}
    \frac{(\sz^{-1} \g z)^{m}}
    {j(\sz^{-1} \g,z)^{k}} \label{ellpoin} \\
   & = &  (2i \beta)^{k/2}\sum_{\g \in   \G}
    \frac{(\g z-z_0)^{m}}
    {(\g z-\overline{z_0})^{m+k}j(\g,z)^{k}}. \nonumber
\end{eqnarray}
Each Poincar\'e series $P_{z_0,m}$ is in $S_k(\G)$ and, for fixed $z_0$ with $m$ varying, they span $S_k(\G)$.
This spanning result is implied by the following analogue of Proposition \ref{porth}, also from
\cite{Petersson:1941a}. See
\cite[\S 4.3]{ImamogluOSullivan:2009} as well.
\begin{prop}  \label{eorth}
Let $z_0\in \H$. Then for $f \in S_k(\G)$  and $m \geqslant 0$ we have
\begin{equation*}
\s{f}{P_{z_0, m}} =  c_{z_0}(f,m) \left[\frac{ \pi (k-2)!
m!}{ 2^{k-3}  (m+k-1)!}\right].
\end{equation*}
\end{prop}
Thus, the  series $P_{z_0, m}$ isolates the $m$-th Fourier coefficient
at $z_0$  in the same manner as $P_{m}$ isolates the
$m$-th Fourier coefficient at the cusp $\ci$. (Note that $P_{z_0, m}(z) = 2\Phi^*_{\text{Ell}}(z,m,z_0)$ in the notation of \cite{ImamogluOSullivan:2009} ).

One of our original motivations  for the work in this paper is the following basic question:
\begin{equation}\label{question}
    \text{ For which $z_0$, $m$ is $P_{z_0,m} \equiv 0$\,?}
\end{equation}
Clearly there is a lot of cancelation in (\ref{ellpoin}) since
$(\sz^{-1} \g z)^m$ circles around $\D$
as $\g$ runs through $\G$.
We prove
the following result about  Poincar\'e series of weight $k=12$ for the modular group:

\begin{theorem} \label{ell lehmer poin}
Let $D \in \mathscr D$, $m \in \Z_{\geqslant 0}$ and assume $2|m$  if $\z_D=\z_{-4}$ and  $3 | m$  if $\z_D=\z_{-3}$. Then each  Poincar\'e series $P_{\z_D,m} \in S_{12}(\G)$ is
not identically zero.
\end{theorem}
\begin{proof}
Proposition \ref{eorth} implies that in $S_{12}(\G)$,
\begin{equation}\label{preln}
P_{z,m} = \overline{c_{z}(\Delta, m)} \left[\frac{  10!
m! \pi}{ 2^{9}  (m+11)!}\right] \frac{\Delta}{\norm{\Delta}^2}.
\end{equation}
Therefore
 $P_{z, m}$ vanishes
identically if and only if $c_{z}(\Delta, m)$ is zero. Hence Theorem \ref{ell lehmer poin} follows from Theorem
\ref{ell lehmer}, and Remark \ref{do they all vanish} also applies to the
vanishing of $P_{z,m}$.
\end{proof}

It is a consequence of Proposition \ref{eorth} and Remark \ref{trivial-vanishing} that $P_{z_0,m}$
necessarily vanishes identically in $S_k(\G)$ when $m \not\equiv -k/2 \bmod N$ if $z_0$ has order $N$. We label such a $P_{z_0,m}$ {\em trivial}.
Petersson proved in \cite{Petersson:1949a} that  the first $d(k)$ non-trivial Poincar\'e series in $\{ P_{z_0,0}, P_{z_0,1}, P_{z_0,2}, \cdots \}$ form a basis for $S_k(\G)$. We see from \eqref{preln} and Theorem \ref{0} that $P_{z_0,2}$ vanishes non-trivially in $S_{12}(\G)$ for $z_0 \approx 1.344 i$.
  We return to  question \eqref{question} in \S \S \ref{numer}, \ref{more on poincareseries}.

\subsection{Non-vanishing of central values of certain
 $L$-functions.}

Lastly, we briefly 
speculate on how this work may be applied to another
very interesting question, namely  non-vanishing of central
critical values of certain character twists of $L$-functions.

The  relation between non-vanishing of Fourier coefficients
at CM points and critical values of certain character twists of
$L$-functions follows from  Waldspurger type
formulas (see e.g. \cite{Waldspurger:1985, Jacquet:1987} and for
more recent developments 
\cite{BruinierJamesKohnenOnoSkinnerVatsal:1999, MartinWhitehouse:2009a, Hida:2010a} and the references therein.)

Let $f \in S_k(\G)$. Waldspurger type formulas give, usually through some type of theta
correspondence, an equality between the absolute value squared of an
adelic twisted toric integral involving $f$ and (a non-zero constant
times) the central critical value of the twisted $L$-function of the
base change $\pi_{f,\K}$ of $f$ to $K=\Q(\z_{D})$, i.e. the value $L(\pi_{f,K}\otimes
\chi^n,1/2)$ where $\chi$ is the basic Gr\"ossencharacter of $K$. In
the present case $\pi_{f,K }$  is the classical Doi-Naganuma
lift.

On the other hand, the adelic twisted toric integral can, in certain
situations,  be related to the $n$-th
Fourier
coefficient of $f$ at CM points like $\z_{D}$ (see
e.g. \cite{Mori:2011, Brakocevic:2011}). So in
these situations we expect that Theorem \ref{ell lehmer} should allow us to show
$$L(\pi_{f,K}\otimes
\chi^n,1/2)\neq 0.$$
Since it is known (\cite{JacquetChen:2001a}) that this value is real and
non-negative we would obtain
$$L(\pi_{f,K}\otimes
\chi^n,1/2)> 0.$$
Reznikov in \cite{Reznikov:2007a} used the Fourier expansion at CM points to
prove sub-convexity of  $L$-functions like $L(\pi_{f,K}\otimes
\chi^n,1/2)$ in the $n$ aspect when $f$ is a Maass form. We hope to return to this topic in a future work.

\section{Fourier coefficients}
\label{basics-fourier-coefficients}

In this section we allow $\G$ to be, more generally, a discrete, finitely generated subgroup  of $\SL(2,\R)$ with $\operatorname{Vol}(\G\backslash \H) < \ci$.
For every $z_0\in
\H\cup\{\textrm{cusps of }\G\}$, each $f \in S_k(\G)$ has a Fourier expansion that we next
describe briefly.

For any $\G' \subset \G$, denote by $\overline{\G'}$ its image under the map $\SL(2,\R) \to \SL(2,\R)/\pm 1$.
Let $\G_{z_0}\subset \G$ be the stabilizer of
$z_0$.  Then $\overline{\G}_{z_0}$ is cyclic and there exists a generator $\g_{0} \in \G$ such that $\overline{\G}_{z_0} = \overline{\langle \g_{0} \rangle}$. If $z_0$
is a cusp  we may
pick a scaling matrix $\sigma_{z_0}\in  \SL(2,\R)$  such that
$$
\sigma_{z_0}^{-1} z_0 = \ci, \quad \sigma_{z_0}^{-1} \g_0\sigma_{z_0}z = z+1.
$$
If $z_0\in \H$ we may choose $\sigma_{z_0}\in \hbox{Isom}^+(\D,\H )$, given by \eqref{scaling} for example, with
$$
\sigma_{z_0}^{-1} z_0 = 0, \quad \sigma_{z_0}^{-1} \g_0\sigma_{z_0}z = \zeta^2 z
$$
and $\zeta^2=j(\g_0,\overline{z_0})^2$ necessarily a primitive $N$-th root of unity for  $1\lqs N=\abs{\overline{\G}_{z_0}}<\infty$. See, for example, \cite[\S 1.2]{Shimura} and \cite[\S \S 2.1, 4.1]{ImamogluOSullivan:2009} for more details.

For $z_0$ a cusp we therefore have the Fourier expansion
\begin{equation}\label{fourierexp2}
f\vert_k\sigma_{z_0}(x+iy)=\sum_{n=-\infty}^\infty
a_{z_0}(f,n)e^{2\pi i n x}
\end{equation}
with $a_{z_0}(f,n)$ depending on $y$. For $z_0 \in \H$ we consider $z \in \D$ in polar coordinates $(r,\theta)$, normalized
so that moving around $0$ once  on a circle corresponds to
$\theta\to \theta+1$. We have the Fourier expansion
\begin{equation}\label{fourierexp}
f\vert_k\sigma_{z_0}\left(re^{2\pi i \theta}\right)=\sum_{n=-\infty}^\infty
a_{z_0}(f,n)e^{2\pi i n \theta} \quad \qquad (z_0 \in \H)
\end{equation}
with
\begin{equation}\label{coe}
a_{z_0}(f,n)=\int_0^1 f\vert_k\sigma_{z_0}\left(re^{2\pi i \theta}\right) e^{-2\pi i n\theta} \,d\theta
\end{equation}
depending on $r<1$.
Rewrite the
$q$-expansion \eqref{af} and the normalized Taylor
expansion \eqref{czfn} as follows:
\begin{eqnarray}
  f\vert_k\sigma_{z_0}(x+iy) &=& \sum_{n=1}^\ci \Bigl[c_{z_0}(f,n) e^{-2\pi n y} \Bigr] e^{2\pi i n x} \quad \text{ for } z_0 \text{ a cusp},\\
  f\vert_k\sigma_{z_0}\left(re^{2\pi i \theta}\right) &=&  \sum_{n=0}^\ci \Bigl[c_{z_0}(f,n) r^{n} \Bigr] e^{2\pi i n \theta} \quad \text{ for } z_0 \in \H.
\end{eqnarray}
Comparing these with \eqref{fourierexp2} and \eqref{fourierexp}  yields formulas for each  $a_{z_0}(f,n)$.
Accordingly,
\begin{equation*}
  a_{z_0}(f,n)=
\begin{cases}
 c_{z_0}(f,n) e^{-2\pi n y} &\textrm{ if \ \ }n \gqs 1 \text{ \ \ and \ \ }z_0 \in \{\textrm{cusps of }\G\}\\
c_{z_0}(f,n)r^n&\textrm{ if \ \ }n \gqs 0 \text{ \ \ and \ \ }z_0 \in \H
  \end{cases}
\end{equation*}
and $a_{z_0}(f,n)=0$ otherwise.
In light of this,  we  label $c_{z_0}(f,n)$ the {\em $n$-th Fourier coefficient $f$ at $z_0$} for all $z_0 \in \H \cup \{\textrm{cusps of }\G\}$, as in the introduction.

Since $f$ is  weight $k$  with respect to $\G$ and $\g_0 \in \G$ we
have
\begin{equation} \label{cor}
  \zeta^{k}f\vert_k\sigma_{z_0}(\zeta^2 z)=\frac{f\vert_k\sigma_{z_0}( \sigma_{z_0}^{-1}
  \g_0\sigma_{z_0}z)}{j(\sigma_{z_0}^{-1}
  \g_0\sigma_{z_0},z)^k}
=f\vert_k\sigma_{z_0}\vert_k\sigma_{z_0}^{-1}
  \g_0\sigma_{z_0}(z)=f\vert_k\sigma_{z_0}( z).
\end{equation}
Using that $\zeta^2=e^{2\pi i m/N}$ where $(m,N)=1$ and inserting \eqref{cor} into \eqref{coe} we obtain
$$
a_{z_0}(f,n) = e^{2\pi i\left(\frac{2 m n+m k}{2N}\right)} a_{z_0}(f,n)
$$
and conclude that unless $(n+k/2)\equiv 0\bmod N$ we
 must have $c_{z_0}(f,n)=0$, as  observed before in Remark \ref{trivial-vanishing}.

We have already seen that $c_{z_0}(f,m)$ can  be expressed as an inner product in Proposition \ref{eorth} and as a polar integral in \eqref{coe}. Since it is also a Taylor coefficient it follows, as in
\cite{Petersson:1941a}, (see also for example \cite[Prop. 16]{ImamogluOSullivan:2009}), that it can be expressed in terms of derivatives of $f$:
\begin{prop}  \label{bim} For $f \in S_k$ and $z_0
\in \H$,
\begin{equation}\label{cz0}
c_{z_0}(f,m)= \sum_{r=0}^m \binom{m+k-1}{r+k-1} \frac{(z_0 - \overline{z_0})^{r+k/2}}{r!}
f^{(r)}(z_0).
\end{equation}
\end{prop}

The Maass raising operator is
\begin{equation}\label{part}
\partial_k f := \dd f - \frac{k}{4 \pi y} f \quad \text{ for } \quad \dd f:= \frac 1{2\pi i} \frac{d}{dz} f(z).
\end{equation}
 It raises the weight by $2$ but does not preserve holomorphy. We follow the convention of writing
$$
\partial^m = \partial_{k+2m-2} \circ \cdots \circ \partial_{k+2} \circ \partial_k
$$
and similarly for $\vartheta$ below. (See (\ref{pppp}) and (\ref{partiale}) in \S \ref{analogs} for the computation of $\partial^m P_n$ and $\partial^m P_{z_0,n}$.)
Use induction to verify that
\begin{equation}\label{dmfz}
\partial^m f(z) = \frac{m!}{(-4 \pi y)^m} \sum_{r=0}^m \binom{m+k-1}{r+k-1} \frac{(2iy)^{r}}{r!}
f^{(r)}(z)
\end{equation}
and comparing \eqref{cz0} with \eqref{dmfz} yields
\begin{equation}\label{czm}
c_{z}(f,m)=(2 \pi i)^m (2iy)^{m+k/2} \frac{\partial^m f(z)}{m!}
\end{equation}
so that $c_{z}(f,m)y^{-k/2-m}$ is a non-holomorphic modular form of weight $k+2m$. The identity \eqref{czm} will be the key to finding formulas for the coefficients $c_{z}(f,m)$ in the next two sections and allows us to prove Theorem \ref{almost all non-vanishing}.

\begin{proof}[Proof of Theorem \ref{almost all non-vanishing}.]
Consider $c_{z}(f,m)$ as a function of $z$. From Proposition \ref{bim}
we see that this function is real analytic. We claim that it does not
vanish identically. Once this has been established real analyticity allows
us to conclude that its set of
zeros has measure zero. Then $A_f$,  as the union of
countably many sets of measure zero, is again of measure zero and the theorem  is proven.

To see that $c_{z}(f,m)$ is non-vanishing we argue
as follows. If we have a cusp (at $\ci$, say) and the $n_0$-th term is the first non-vanishing term
in the Fourier expansion of $f$ at $\ci$ then $f^{(r)}/f\to
(2\pi i n_0)^r$ as $y\to\infty$. From this and Proposition
\ref{bim} it easily follows that  $c_{z}(f,m)$ is not identically
zero. If $\G$ does not have any cusps we assume that $c_{z_0}(f,m)$ vanishes identically. Identity \eqref{czm} implies that $c_{z_0}(f,n) \equiv 0$ also for all $n > m$. Thus
$$
f(z)=\frac{(z_0-\overline{z_0})^{k/2}}{(z-\overline{z_0})^{k}} \sum_{n=0}^{m-1} c_{z_0}(f,n)\left( \frac{z-z_0}{z-\overline{z_0}}\right)^n
$$
implying in particular that $f(z)$, as a rational function, has a finite number of zeros in $\H$. But the results \cite[Prop. 2.16, Theorem 2.20]{Shimura} imply that $\deg(\operatorname{div}(f)) = k \operatorname{Vol}(\G\backslash \H)/(4\pi)$ and hence that the holomorphic $f$ has at least one zero in $\G\backslash \H$. Therefore $f$ has infinitely many zeros in $\H$ and we have contradicted our assumption that $c_{z_0}(f,m)$ vanishes identically.
\end{proof}


\section{Recursive formulas for Fourier coefficients at non-cuspidal
  points} \label{non-vanishing}

We review the theory of Rodriguez Villegas and Zagier which in certain
arithmetic cases allows us to compute the coefficients $c_{z_0}(f,m)$ explicitly. See \cite{VillegasZagier:1993a}, and \cite[p. 50 - 55, p. 88]{BruinierGeerHarderZagierRanestad:2008a} in particular.
To proceed further, we  focus exclusively on the group $\Gamma=\SL(2,\Z)$ in \S \S \ref{non-vanishing},  \ref{cma}, \ref{arit} and \ref{numer}.
 Let
$$
E_2(z):=1-24\sum_{n \in \Z_{\gqs 1}}\sigma_1(n)e^{2\pi i n z}
$$
be the holomorphic quasimodular Eisenstein series and put $E^*_2(z)=E_2(z)-3/(\pi y)$ so that $E^*_2$ has weight $2$ but is not holomorphic.
The weight $k>2$ holomorphic
Eisenstein series is
$$
E_k(z):=\sum_{\g \in \G_\infty \backslash \G} \left( 1|_k \g \right)(z) = \frac{1}{2}\sum_{\substack{c,d\in \Z\\
    (c,d)=1}}\frac{1}{(cz+d)^k}.
    $$
To aid clarity we will often use the alternate notation $Q$ for $E_4$ and $R$ for $E_6$.

A variation of (\ref{part}), defining the Maass raising operator $\partial$, gives the Serre derivative
$$
\vartheta f = \vartheta_k f := \dd f - \frac{k}{12}E_2  f
$$
mapping $S_k(\G) \to S_{k+2}(\G)$. Our next goal is to show:

\begin{theorem} \label{taylor} For $m \gqs 0$ the $m$-th Fourier coefficient of $\Delta$ at any point $z \in \H$ is
$$
c_{z}(\Delta,m)=\Delta(z)\left(\frac{\pi i}{6}\right)^m (z-\overline{z})^{m+6} \sum_{r=0}^m \frac{1}{r!} \binom{m+11}{r+11} \left(E^*_2\right)^{m-r}\mathcal B_{r}
$$
with $\mathcal B_r \in M_{2r}(\G)$  defined recursively as
\begin{equation}\label{pn1}
\mathcal B_0=1, \ \ \mathcal B_1 = 0, \ \ \mathcal B_{n+1}=12\vartheta \mathcal B_n - n(n+11) Q \mathcal B_{n-1} \qquad (n \gqs 1).
\end{equation}
\end{theorem}
(Computing, we find for example that $\mathcal B_2 = -12Q$, $\mathcal B_3=48R$, $\mathcal B_4=216 Q^2$, $\mathcal B_5=-4608 QR$ and $\mathcal B_6=1152(9Q^3+16R^2)$.)
\begin{proof}
Rodriguez Villegas and Zagier recursively define a modified
Serre derivative that also stays in the
space of holomorphic cusp forms:
\begin{equation}\label{modified_serre}
\vartheta^{[0]}f = f, \ \ \vartheta^{[1]}f = \vartheta f, \ \ \vartheta^{[n+1]}f = \vartheta ( \vartheta^{[n]} f) -n(n+k-1) \frac{E_4}{144} \vartheta^{[n-1]}f.
\end{equation}
We have $\vartheta^{[n]}: S_k(\G) \to S_{k+2n}(\G)$.
Then an induction argument (or combining equalities (56) and (65) in \cite{BruinierGeerHarderZagierRanestad:2008a}) yields
\begin{equation}\label{partmf}
\partial^m f = \sum_{r=0}^m \frac{m!}{r!} \binom{m+k-1}{r+k-1} \left(\frac{E^*_2}{12}\right)^{m-r}
\vartheta^{[r]}f.
\end{equation}
To compute $\vartheta^{[n]} f$ explicitly we write it as a polynomial in $Q$ and $R$. Recall that any element of $S_{k+2n}(\G)$ may be expressed as a linear combination of terms of the form $R^a Q^b$ with $6a+4b=k+2n$. Set
$$
\mathcal A_n(f) := 12^n \vartheta^{[n]} f
$$
with $\mathcal A_n(f)$ a polynomial in $Q$, $R$ as described above. The factor $12^n$ is included to ensure that we obtain integer coefficients.
Then, using results going back to Ramanujan \cite[p. 88]{BruinierGeerHarderZagierRanestad:2008a},
we have
\begin{equation}\label{ramanujan-identities}
\dd E_2=\frac{E_2^2-Q}{12},\quad
\dd Q=\frac{E_2Q-R}{3},\quad
\dd R=\frac{E_2R-Q^2}{2},
\end{equation}
so that
$$
\vartheta \mathcal A_n(f) = \left(\frac{-R}{3}\frac{\partial}{\partial Q} - \frac{Q^2}{2}\frac{\partial}{\partial R} \right) \mathcal A_n(f)
$$
and we obtain the recursion
$$
\mathcal A_0(f)=f, \ \ \mathcal A_1(f) = 12\vartheta f, \ \ \mathcal A_{n+1}(f)=12\vartheta \mathcal A_n(f) - n(n+k-1) Q \mathcal A_{n-1}(f).
$$
We are interested in the special case $f=\Delta = (Q^3-R^2)12^{-3}$. We have $\vartheta \Delta =0$ since it is in $S_{14}(\G)$ and it follows that $\Delta$ is a factor of every $\mathcal A_n(\Delta)$. Write $\mathcal A_n(\Delta)=\Delta \cdot \mathcal B_n$ with $\mathcal B_n$ necessarily satisfying the recursion \eqref{pn1}.
 We obtain the theorem  with (\ref{czm}), (\ref{partmf}) and the identity
 $$
  \vartheta^{[n]}\Delta = 12^{-n} \Delta \cdot \mathcal B_n.
 $$
\end{proof}

To write our results in terms of a polynomial in only one variable, see \cite[p. 88]{BruinierGeerHarderZagierRanestad:2008a},
we observe that
every term in $\mathcal B_n$ contains a product of the form
$$
R^a Q^b = \left(\frac{R}{Q^{3/2}} \right)^a Q^{n/2}.
$$
So, assuming that $Q \neq 0$, we may write
\begin{equation}\label{pn2}
\mathcal B_n = Q^{n/2} p_n(R Q^{-3/2}).
\end{equation}
Substituting (\ref{pn2}) into (\ref{pn1}), we see that $p_n(t) \in \Z[t]$ satisfies the recursion
\begin{equation}\label{prec}
p_0=1, \ \ p_1 = 0, \ \ p_{n+1} = - 2n t p_n + 6(t^2-1) p'_n - n(n+11) p_{n-1} \qquad (n \gqs 1).
\end{equation}

Alternatively, for $R \neq 0$
write each term
$
Q^a R^b = \left(\frac{Q}{R^{2/3}} \right)^a R^{n/3}
$
so that
\begin{equation} \label{pn3}
\mathcal B_n = R^{n/3} q_n(Q R^{-2/3})
\end{equation}
with $q_n(t)\in \Z[t]$ satisfying the recursion
\begin{equation}\label{qrec}
q_0=1, \ \ q_1 = 0, \ \ q_{n+1} = - 2n t^2 q_n + 4(t^3-1) q'_n - n(n+11)t q_{n-1} \qquad (n \gqs 1).
\end{equation}

\section{Fourier developments at CM points} \label{cma}
Let $D$ be a negative fundamental discriminant as before and recall the notation $\mathfrak z_D$ from \eqref{zzd}. Define the Chowla-Selberg period
$$
\Omega_D:=\frac{1}{\sqrt{2\pi |D|}} \left[ \prod_{j=1}^{|D|-1} \G(j/|D|)^{\left(\frac Dj\right)}\right]^{\frac 1{2h(D)}}
$$
for $D<-4$ and also for $D=-3,-4$ but with the class number $h(D)$ replaced by $1/3$, $1/2$ respectively. Clearly $\Omega_D \in \R_{>0}$.

\subsection{Fourier expansion of $\Delta$ at $i$}\label{i}
We are in the case $D=-4$ with $\mathfrak z_{-4} = i$ and
$$
\Omega_{-4}=\frac1{2\sqrt{2\pi}}\left[\frac{\G(1/4)}{\G(3/4)}\right] = \frac{\G(1/4)^2}{4\pi^{3/2}} \approx 0.5902.
$$
Using the table on p. 87 of \cite{BruinierGeerHarderZagierRanestad:2008a}\footnote{The table entry $(3-\sqrt{5})/2$ for $D=-15$ should read $(-3+\sqrt{5})/2$.},
$$
E^*_2(i)=0, \ \ Q(i)= 12 \Omega_{-4}^4, \quad R(i)= 0, \quad \Delta(i)=\Omega_{-4}^{12}.
$$
Therefore, with Theorem \ref{taylor} and (\ref{pn2})
\begin{equation}\label{expr at i}
c_i(\Delta,n)= -|D|^3 \Delta(i)\left(\frac{-2\pi \Omega_{D}^{2}}{\sqrt{3}}\right)^n  \frac{p_n(0)}{n!}  \qquad (n \gqs 0)
\end{equation}
with $p_{n}(0) \in \Z$.
For example, with $0\leqslant n <12$ the numbers $p_{n}(0)$ are
$$
1,\ 0,\  -12,\ 0,\  216,\ 0,\  10368,\ 0,\  -2052864,\ 0,\  47029248,\ 0.
$$

We are already now in a position to prove part of Theorem \ref{ell lehmer}: that $c_i(\Delta,n)\neq 0$ for all  even $n$.
\begin{prop}\label{mod5}
With $p_{n}(t)$ defined as above by (\ref{prec}) we have
\begin{eqnarray*}
  p_{2m}(0) & \equiv & 1\mod 5 \quad \text{ for $m$ even}, \\
  p_{2m}(0) & \equiv & 3\mod 5 \quad \text{ for $m$ odd}.
\end{eqnarray*}
\end{prop}
\begin{proof}
Calculating the polynomials $p_n(t)$ modulo $5$ for $0 \leqslant n <20$, we find:
\begin{equation}\label{arr}
\begin{array}{llll}
1, & 0, & 3, & 3 t, \\
1, & 2 t, & 3 + 2 t^2, \ \ \ & t, \\
1, & 2 t, & 3 +  t^2, & 3 t + 2 t^3, \\
1 + 4 t^2 + 2 t^4, \ \ \ & 2 t^3, & 3 +  t^2, & 4 t + 4 t^3, \\
1 + 2 t^2 + 2 t^4, & t + 4 t^3 + 4 t^5, \ \ \ & 3 + 3 t^2 + 4 t^4 + 4 t^6, & 4 t + 4 t^3.
\end{array}
\end{equation}
The recursion (\ref{prec}) modulo $5$ is
\begin{equation}\label{recl}
p_{n+1}(t) \equiv -2n t p_n(t) +(t^2-1)p'_n(t) -n(n+1)p_{n-1}(t)
\end{equation}
and we see that it only depends on $n \bmod 5$. Computing the next two terms we find
$p_{20}(t) \equiv 1$ and  $p_{21}(t) \equiv 0$. Thus we must have $p_{n+20}(t) \equiv p_n(t) \bmod 5$ and the result follows.
\end{proof}

\begin{remark}\label{simplification} {\em A simplification we will need later is possible. From the recursion \eqref{recl} we see that the constant term of $p_{n+1}(t)$ depends on the constant terms of $p_{n}(t)$ and $p_{n-1}(t)$. It may also depend on the coefficient of $t$ in $p_{n}(t)$ because of the term $-p_n'(t)$ in \eqref{recl}. In turn the coefficient of $t$ may depend on a previous coefficient of $t^2$, and so on. But if we are working mod $5$ the coefficients of $t^5$ and higher powers of $t$ cannot affect the constant term of subsequent $p_{n}(t)$s because $(d/dt) t^5 \equiv 0 \bmod 5$. Thus the terms $4 t^5$ and $4 t^6$ in \eqref{arr} do not play a role and may be ignored.}
\end{remark}
Accordingly, we define the ring
\begin{equation}\label{ring}
 \mathcal R_l:= (\Z/l\Z)[t]/ (t^l)
\end{equation}
 of polynomials of degree at most $l-1$ with coefficients in $\Z/l\Z$. Denote the natural projection $\Z[t] \to \mathcal R_l$ by $p \mapsto \overline p$.
Thus we have proved that if we define a recursion
\begin{equation*}
p^*_0=1, \ \ p^*_1 = 0, \ \ p^*_{n+1} \equiv - 2n t p^*_n + 6(t^2-1) (p^*_n)' - n(n+11) p^*_{n-1}
\end{equation*}
in $\mathcal R_l$ then
$$
\overline{p}_n \equiv p^*_n \text{ \ in \ } \mathcal R_l
$$
and
$$
p_n(0) \equiv \overline{p}_n(0) \equiv p^*_n(0) \mod l.
$$

\subsection{Fourier expansion of $\Delta$ at $\omega$} \label{w}
We have $D=-3$ with $\mathfrak z_{-3} =(1+i \sqrt{3})/2 = \omega+1$ and
$$
\Omega_{-3}=\frac1{\sqrt{6\pi}}\left[\frac{\G(1/3)}{\G(2/3)}\right]^{3/2} = \frac{\root 4 \of 3 \G(1/3)^3}{4\pi^{2}} \approx 0.6409.
$$
Also
$$
E^*_2(\omega)=0, \quad Q(\omega)= 0, \quad R(\omega)= 24\sqrt{3} \Omega_{-3}^6, \quad \Delta(\omega)=-\Omega_{-3}^{12}.
$$
With Theorem \ref{taylor} and (\ref{pn3})
\begin{equation}\label{expr at w}
c_\omega(\Delta,n)= -|D|^3 \Delta(\omega)\Bigl(-\pi \Omega_{D}^{2}\Bigr)^n  \frac{q_n(0)}{n!} \qquad (n \gqs 0)
\end{equation}
with $q_{n}(0) \in \Z$.
For example, with $0\leqslant n <15$ the numbers $q_{n}(0)$ are
$$
1,\ 0,\ 0,\ 48,\ 0,\ 0,\ 18432,\ 0,\ 0,\ 13271040,\ 0,\ 0,\ 1974730752,\ 0,\ 0.
$$
\begin{prop} \label{mod7}
With $q_{n}(t)$ defined as above by  (\ref{qrec}), we have
\begin{eqnarray*}
  q_{3m}(0) & \equiv & 1\mod 7 \quad \text{ for $m$ even}, \\
  q_{3m}(0) & \equiv & 6\mod 7 \quad \text{ for $m$ odd}.
\end{eqnarray*}
\end{prop}
\begin{proof}
The polynomials $\overline{q}_{n}(t)$ in $\mathcal R_7$ for $0 \leqslant n < 42$ are
\begin{equation}\label{arr2}
\begin{array}{llllll}
\x 1, & \x 0, & \x 2t, & \x 6, & \x 6t^2, & \x 5 t, \\
\x 1 + t^3, & \x 5 t^2, & \x 2 t + 5 t^4, & \x 6 + 4 t^3, & \x 2 t^2, & \x 5 t + 4 t^4, \\
\x 1 + 6 t^3 + 4 t^6, & \x t^2 + 2 t^5, & \x 2 t + 2 t^4, & \x 6 + 4 t^3 + 4 t^6, & \x 4 t^5 + 4 t^8, & \x 5 t + 5 t^4, \\
\x 1 + t^3 + t^6, & \x 2 t^2 + 2 t^5, & \x 2 t + 2 t^4, & \x 6, & \x 0, & \x 5 t, \\
\x 1, & \x t^2, & \x 2 t, & \x 6 + 6 t^3, & \x 2 t^2, & \x 5 t + 2 t^4, \\
\x 1 + 3 t^3, & \x 5 t^2, & \x 2 t + 3 t^4, & \x 6 + t^3 + 3 t^6, & \x 6 t^2 + 5 t^5, & \x 5 t + 5 t^4, \\
\x 1 + 3 t^3 + 3 t^6, & \x 3 t^5 + 3 t^8, & \x 2 t + 2 t^4, & \x 6 + 6 t^3 + 6 t^6, & \x 5 t^2 + 5 t^5, & \x 5 t + 5 t^4.
\end{array}
\end{equation}
We find that $\overline{q}_{42}(t) \equiv \overline{q}_0(t)$ and $\overline{q}_{43}(t) \equiv \overline{q}_1(t)$ in $\mathcal R_7$. Therefore $\overline{q}_{n+42}(t) \equiv \overline{q}_n(t)$ and the proof is complete.
\end{proof}

It follows from Proposition \ref{mod7} that all non-trivial Fourier coefficients of $\Delta$ at $\omega$ are non-zero. The modulus $l=7$ was the smallest possible to show all $q_{3m}(0) \not \equiv 0 \bmod l$, but we could also have used $l=13$, $19$ and $43$. Similarly $p_{2m}(0) \not \equiv 0 \bmod l$ for $l=13,37,41$  as well as $5$ above. Note that the  $5$  rows in the array \eqref{arr} and the $7$ rows in \eqref{arr2} are examples of the pattern we see in Proposition \ref{qper}.

\subsection{Fourier expansion of $\Delta$ at other CM points}
Let $\z=\z_D$ be a CM point of discriminant $D <-4$. We continue with the method outlined by Zagier in \cite[p. 88]{BruinierGeerHarderZagierRanestad:2008a}.
 Using the
results and normalization from \cite[p. 86 - 87]{BruinierGeerHarderZagierRanestad:2008a}, we may let
$$
E_2^*(\z)=k_1\abs{D}^{-1/2}\Omega_{D}^2, \quad
Q(\z)=k_2\Omega_{D}^4,\quad R(\z)=k_3\abs{D}^{1/2}\Omega_{D}^6
$$
for non-zero $k_1, k_2, k_3$ in some number field $K = K_\z$. There exists $k_0 \in \Z_{>0}$ such that
$$
k_0k_1, \ k_0k_2, \ k_0 k_3 |D| \ \in \
\mathcal{O}_K.
$$
Choose now $m_1, m_2\in \mathcal{O}_K$ such that
\begin{equation}
  \label{should-have-zero}
  \left(E_2^*-\frac{m_1}{m_2}\frac{R}{Q}\right)(\z)=0.
\end{equation}
To make the choice definite, set  $m_1=k_0^2 k_1k_2$ and
$m_2=k_0^2\abs{D}k_3$, but with any common factors in $\Z_{>1}$ removed.
Lastly, define these elements of $\O_K[t]$:
\begin{eqnarray*}
  a_1(t) &:=& 12 k_0^4 k_2 m_1 \abs{D}(t+k_3), \\
  a_2(t) &:=& 2 k_0^4 k_2 (m_1-m_2) \abs{D}(t+k_3), \\
  a_3(t) &:=& 6 k_0^4  m_2 \bigl( k_2 \abs{D}(t+k_3)^2-k_2^4\bigr), \\
  a_4(t) &:=& -k_0^8 k_2^5 m_2 (m_2-6m_1)\abs{D}-k_0^8 k_2^2 m_1 (4m_2+m_1)\abs{D}^2(t+k_3)^2.
\end{eqnarray*}
We are now ready to state
\begin{theorem} \label{cmpoint}
For a CM point $\z=\z_D$ with $D<-4$, define $K$, $k_i$, $m_i$ and $a_i(t)$ as above. We have
\begin{equation}\label{q-to-c}
c_{\z}(\Delta,n)=-\abs{D}^3 \Delta(\z) \left(\frac{-\pi \Omega_D^{2}}{6k_0^4 k_2^2 m_2 }\right)^n \frac{q_n(0)}{n!} \qquad (n \gqs 0)
\end{equation}
with $q_n \in \O_K[t]$ satisfying the recurrence
\begin{equation}\label{rqe}
\begin{split}
q_{0, \z}(t)={}& 1, \quad q_{1, \z}(t)=a_{1}(t), \\
q_{n+1,\z}(t)={}&(a_{1}(t)+n \,a_{2}(t))q_{n,\z}(t)+a_{3}(t)q'_{n,\z}(t)+n(n+11)a_{4}(t)q_{n-1,\z}(t) \qquad(n \gqs 1).
\end{split}
\end{equation}
\end{theorem}
\begin{proof}
We let
$\phi:=\frac{1}{12}\left(E_2-\frac{m_1}{m_2}\frac{R}{Q}\right)$,
which we note is holomorphic except for a pole at the  zeros of $Q$,
i.e. at $\omega$, the third root of unity. Define
$\w:=\dd \phi-\phi^2$.
Using (\ref{ramanujan-identities}) it is
straightforward to verify that in terms of $E_2, Q$, and  $R$ the
function $\w$ is
given by
\begin{equation}\label{omega}
  \w=\frac{-Q}{144m_2^2}\left(m_2(m_2-6m_1)+(4m_1m_2+m_1^2)(R/Q^{3/2})^2\right).
\end{equation}

We next define a
derivative $\vartheta_{k,\phi}$ from \cite[p. 88]{BruinierGeerHarderZagierRanestad:2008a} by
 \begin{equation*}
  \vartheta_{k,\phi}f=\vartheta_{\phi}f:=\dd f-k\phi f.
\end{equation*}
 It is a further modification of Rodriguez Villegas and Zagier's \eqref{modified_serre} giving a mapping from the set of meromorphic modular forms of
weight $k$ to the space of meromorphic modular forms of weight $k+2$.

We note that $\vartheta_{k,\phi}f=\vartheta f +Vf$ where $V$ is
multiplication by $\frac{km_1}{12m_2}\frac{R}{Q}$. We then define
 a meromorphic modular form $\vartheta_\phi^{[n+1]}f$   of weight
 $k+2n$ from the recursion
\begin{equation*}
\vartheta_\phi^{[0]}f = f, \ \ \vartheta_\phi^{[1]}f = \vartheta_\phi f, \ \ \vartheta_\phi^{[n+1]}f = \vartheta_\phi ( \vartheta_\phi^{[n]} f) +n(n+k-1) \w \ \vartheta_\phi^{[n-1]}f.
\end{equation*}

\begin{lemma} \label{recurrence} For $z$ not a zero of $Q$ the following holds:
  \begin{equation*}
    (12m_2)^n\vartheta_\phi^{[n]} \Delta(z)=\Delta(z) Q^{n/2}(z) p_n\bigl(R(z)/Q(z)^{3/2}\bigr),
  \end{equation*}
where the polynomials $p_n(t)\in \O_K[t]$ are given by the recursion
\begin{align*}
p_0(t)&=1, \ \ p_1(t) =12 m_1 t , \\ p_{n+1}(t) &= (2(m_1-m_2)n+12m_1)tp_n(t)+6m_2(t^2-1)p_n'(t)\\& \quad -n(n+11)\Bigl[m_2(m_2-6m_1)+(4m_1m_2+m_1^2)t^2\Bigr] p_{n-1}(t).
\end{align*}
\end{lemma}
\begin{proof}
This is a straightforward (if long) induction using that $\vartheta_\phi$ is a
derivation, (\ref{omega}), and the identities
\begin{equation*}
  \vartheta_\phi(\Delta)=(\vartheta+V)(\Delta)=V\Delta=\frac{m_1}{m_2}Q^{1/2}\frac{R}{Q^{3/2}}\Delta,
\end{equation*}
\begin{equation*}
  \vartheta_\phi(Q)=Q^{3/2}\frac{m_1-m_2}{3m_2}\frac{R}{Q^{3/2}},
\end{equation*}
\begin{equation*}
  \vartheta_\phi(R/Q^{3/2})=\frac{Q^{1/2}}{2}\left(-1+\left(\frac{R}{Q^{3/2}}\right)^2\right).
\end{equation*}
\end{proof}
The generalization of \eqref{partmf} is then
\begin{equation*}\label{partmf2}
\partial^m f(z) = \sum_{r=0}^m \frac{m!}{r!} \binom{m+k-1}{r+k-1} \left(\frac{E_2^*}{12}-\frac{m_1 R}{12 m_2 Q}\right)^{m-r}
\vartheta^{[r]}_\phi f(z)
\end{equation*}
and at $z=\z$ we
obtain, by the requirement (\ref{should-have-zero}),
\begin{equation}\label{only-one-term}(\partial^n f)(\z)= (\vartheta_\phi ^{[n]}f )(\z).
\end{equation}

We have $R(\z)/Q(\z)^{3/2}=\frac{\sqrt{k_2 \abs{D}}}{k_2^2}k_3$.
It
turns out to be convenient to change to a different set of polynomials,
namely
\begin{equation}\label{k00}
q_n(t):=\Bigl(k_0^4 k_2^2 \sqrt{k_2\abs{D}}\Bigr)^n p_n\left(\frac{\sqrt{k_2\abs{D}}}{k_2^2}(t+k_3)\right)
\end{equation}
so that
\begin{equation}
\label{translated-to-zero}q_n(0)=\Bigl(k_0^4 k_2^2 \sqrt{k_2\abs{D}}\Bigr)^n p_n(R(\z)/Q(\z)^{3/2}).
\end{equation}
Since
$p_n(t)$ has degree at most $n$ and contains only even powers of $t$ if $n$ is even and only odd
powers if $n$ is odd, we see (also using the definition of  $k_0$) that $q_n(t) \in \O_K[t]$. (In fact $k_0^4$ may be replaced by $k_0^3$ in \eqref{k00} and we still have $q_n(t) \in \O_K[t]$, but it complicates the exposition slightly.)

The recurrence relations of $p_n$ from Lemma \ref{recurrence} translate into relations for $q_n$ in \eqref{rqe}.
From (\ref{czm}), (\ref{only-one-term}), Lemma \ref{recurrence},
and (\ref{translated-to-zero}) we deduce \eqref{q-to-c}.
\end{proof}

As a consistency check, we may  verify that Theorem \ref{cmpoint} agrees with Proposition \ref{bim}.  We obtain
\begin{equation}\label{test}
c_{z}(\Delta,m)= (2i y)^6\sum_{n=1}^\infty \tau(n) q^n \sum_{r=0}^m \binom{m+11}{r+11} \frac{(-4\pi n y)^{r}}{r!}
\end{equation}
easily from Proposition \ref{bim}. Then with $D=-20$, for example, \eqref{q-to-c} and \eqref{test} both give numerically
$$
\Delta|_{12}\sigma_{\z_{-20}}(z) = -0.0063+0.1019z -0.6803z^2+ 2.3012z^3 -3.4187z^4+O(z^5).
$$
See Example \ref{d20a} below.

We may now prove our main result.
\begin{proof}[Proof of Theorem \ref{mainthm}]
Assembling our formulas from this section we see that in the case $D=-4$ we have $\z=\z_D=i$ and
\begin{equation}\label{tyb}
c_\z(\Delta,m) = \kappa_{\z}\frac{\lambda_{\z}^m}{m!}q_{m, \z}(0) \qquad(m \gqs 0)
\end{equation}
for  $\kappa_{\z}=-\abs{D}^3 \Delta(i)$, $\lambda_{\z}=-2\pi \Omega_D^2/\sqrt{3}$, by \eqref{expr at i}, and $q_{m, \z}(t) \in \O_K[t]$ with $K=\Q$ given by the recursion \eqref{simplified-writing} for $a_1=0$, $a_2=-2t$, $a_3=6(t^2-1)$ and $a_4=-1$ using \eqref{prec}.

In the case $D=-3$ we have $\z=\z_D=\omega+1$ and \eqref{tyb} holds
for  $\kappa_{\z}=-\abs{D}^3 \Delta(\omega)$, $\lambda_{\z}=-\pi \Omega_D^2$, by \eqref{expr at w}, and $q_{m, \z}(t) \in \O_K[t]$ with $K=\Q$ given by the recursion \eqref{simplified-writing}  with $a_1=0$, $a_2=-2t^2$, $a_3=4(t^3-1)$ and $a_4=-t$ using \eqref{qrec}.

Finally, for $D$ a negative fundamental discriminant $< -4$, (and hence $\z_D$ not a zero of $Q$ or $R$), \eqref{tyb} holds for $\z=\z_D$ by Theorem \ref{cmpoint} with $\kappa_{\z}=-\abs{D}^3 \Delta(\z)$, $\lambda_{\z}= -\pi \Omega_D^{2}/(6k_0^4 k_2^2 m_2)$ and the other constants as defined there.
\end{proof}

\subsection{Examples}
We demonstrate how Theorem \ref{cmpoint} may be used to prove non-vanishing of $c_{\z_{D}}(\Delta,n)$, in the same manner as Propositions \ref{mod5} and \ref{mod7},  for the seven remaining fundamental discriminants $D \in \mathscr D$.
Generalizing \eqref{ring}, let
\begin{equation}\label{rlz}
\mathcal R_l = \mathcal R_{l,\z}:=(\O_K/l \O_K)[t]/(t^l)
\end{equation}
 and denote the natural projection from  $\O_K[t] \to\mathcal R_l$ by $p \mapsto \overline p$.

\begin{example}\label{d7}
{\em
$D=-7$: In this case we
 have $k_1=3$, $k_2=15$ and $k_3=27$. Then $m_1=5$,
$m_2=21$, $k_0=1$ and $K=\Q$.  The
recursion  for the  $q_n$  polynomials becomes
$q_0=1$, $q_1(t)=60\cdot 105\cdot (t+27)$,
\begin{align*}
q_{n+1}(t)=&105(-32n+60)(t+27)q_n(t)\\ & +126\cdot
(105(t+27)^2-15^4)q'_n(t)\\ &-n(n+11)105(-19845\cdot 15^4+445\cdot105^2\cdot 105\cdot(t+27)^2)q_{n-1}(t).
\end{align*}
Computing a few values we suspect that $q_n(0)\not\equiv 0 \bmod l$ when
$l=23$.
 Further calculations reveal that
\begin{align*}
\overline q_{265}(t) = \ & 8 \overline  q_{12}(t)=17+20t+3t^2+3t^3+16t^4\\
\overline q_{266}(t) = \ & 8 \overline  q_{13}(t)=13+6t+7t^2+9t^3+12t^4+19t^5.
\end{align*}
 This implies, since $265\equiv 12 \bmod l$,  that
\begin{equation}\label{265}
\overline
q_{n+(265-12)}(t)=8\overline q_n(t) \quad \text{for all} \quad n\gqs 12
\end{equation}
 and, since $8
\bmod 23$ is of order $11$,  that
$\overline{q}_{n+11(265-12)}(t)=\overline q_{n}(t)$
for $n\gqs 12$. It follows that $q_{n+2783}(0) \equiv q_n(0) \bmod 23$ for $n\gqs 12$ and a
computation shows that it holds also for $n=1,\ldots,11$.  On inspection we find
that   $q_n(0)\not\equiv 0 \bmod 23$ for $n<265$, and by \eqref{265}
this holds for all $n$, so we have proved that for every $n \in \Z_{\gqs 0}$
$$c_{\z_{-7}}(\Delta,n)\neq 0.$$
Working modulo $l=43$, $67$ or $79$ gives a similar proof of non-vanishing.
}
\end{example}

\begin{example}
{\em
$D=-8$: We have $k_1=4$,
$k_2=20$, $k_3=28$, with $m_1=5$, $m_2=14$, $k_0=1$, and $K=\Q$.  After some experimenting we guess that we should
look at $l=17$. We find that
\begin{align*}\overline q_{550}(t) =&2\overline  q_{278}(t)=6+11t+9t^2\\ \overline q_{551}(t) =&2\overline  q_{279}(t)=
 15+3t+5t^2+12t^3 \end{align*} which implies, since $550\equiv 278 \bmod l$, that
$\overline q_{n+(550-278)}(t)=2\overline q_n(t)$ for $n\gqs 278$. And since 2
mod 17 is of order 8 we have
$$\overline{q}_{n+8(550-278)}(t)=\overline q_{n}(t), $$
for $n\gqs 278$. It follows that $q_{n+2176}(0) \equiv q_n(0) \bmod 17$ for $n\gqs 278$, and by
inspection we see that it holds indeed for all $n>0$. By computing
$q_n(0) \bmod 17$ for every $n < 550$ and observing that these
values are all different from zero we conclude that $q_n(0)\not \equiv 0 \bmod 17$ for
all $n\gqs 0$. It follows that $c_{\z_{-8}}(\Delta,n)\neq 0$ for all $n$.
}
\end{example}

\begin{example}
{\em
$D=-11$: We have $k_1=8$,
$k_2=32$, $k_3=56$, $m_1=32$, $m_2=77$, $k_0=1$, and $K=\Q$. Computing mod $l=23$ we obtain
$\overline q_{n+253}(t) =14\overline
  q_{n}(t)$ for $n \gqs 12$ which implies, since $14
\bmod 23$ is of order $22$,   that $q_{n+5566}(0) \equiv q_n(0) \bmod 23$ for $n\gqs 12
$, and as before we may verify that $q_n(0)\not \equiv 0 \bmod 23$ and $c_{\z_{-11}}(\Delta,n)\neq 0$ for all $n$.
}
\end{example}

\begin{example}
{\em
$D=-15$: We have $k_1=6+3\sqrt{5}$,
$k_2=15+12\sqrt{5}$, $k_3=42+63/\sqrt{5}$ with
$m_1=13+30\sqrt{5}$, $m_2=70+21\sqrt{5}$, $k_0=1$, and
$K=\Q(\sqrt{5})$. Reducing mod $17 \O_K$ we find
$\overline{q}_{278}(t)=(13+10\sqrt{5})\overline{q}_{6}(t)$,
$\overline{q}_{279}(t)=(13+10\sqrt{5})\overline{q}_{7}(t),$ and
$278\equiv 6 \bmod 17$.   The
number $(13+10\sqrt{5})$ has order 144 mod $17 \O_K$
so $\overline{q}_n(t)=\overline q_{n+39168}(t)$  for $n \gqs 6$. We have $q_n(0) \not \equiv 0
\bmod 17 \O_K$ for all $n < 278$ so it follows  that $q_n(0) \not \equiv 0
\bmod 17 \O_K$ for all $n\gqs 0$.  Hence $c_{\z_{-15}}(\Delta,n)\neq 0$ for all $n$.
}
\end{example}

Recall that for $K=\Q(\sqrt{5})$ we have $\O_K = \Z\left[(1+\sqrt{5})/2\right]$. In the above example we are implicitly using the isomorphism
$$
\Z/l\Z \left[\sqrt{5}\right] \to \O_K/l\O_K \quad \text{given by} \quad a+b \sqrt{5} \mapsto a+b \sqrt{5} +l\O_K \qquad \text{($l$ odd)}.
$$

\begin{example}
{\em
$D=-19$: We have $k_1=8$,
$k_2=32$, $k_3=56$ with $m_1=32$, $m_2=57$, $k_0=1$, and $K=\Q$.
   If we reduce mod 7 we
find $\overline{q}_{21}(t)=4\overline{q}_{0}(t)=4$,
$\overline{q}_{22}(t)=4\overline{q}_{1}(t)=2+5t$, $21\equiv 0 \bmod l$, and 4 has order 3,
so $\overline{q}_n(t)=\overline q_{n+63}(t)$. Arguing as before, $c_{\z_{-19}}(\Delta,n)\neq 0$ for all $n$.
}
\end{example}

\begin{example} \label{d20a}
{\em
$D=-20$: We have $k_1=12+4\sqrt{5}$,
$k_2=40+12\sqrt{5}$, $k_3=72+63/\sqrt{5}$ so that
$m_1=45+19\sqrt{5}$, $m_2=90+28\sqrt{5}$, $k_0=1$, and
$K=\Q(\sqrt{5})$. Reducing mod $7 \O_K$ we find
$\overline{q}_{21}(t)=(4+6\sqrt{5})\overline{q}_{0}(t)$,
$\overline{q}_{22}(t)=(4+6\sqrt{5})\overline{q}_{1}(t)$, and $21\equiv 0 \bmod 7$.  The
number $(4+6\sqrt{5})$ has order 24 mod $7 \O_K$
so $\overline{q}_n(t)=\overline q_{n+504}(t)$. We have $q_n(0) \not \equiv 0
\bmod 7 \O_K$ for all $n < 21$ so it follows from the periodicity that $q_n(0) \not \equiv 0
\bmod 7 \O_K$ and $c_{\z_{-20}}(\Delta,n)\neq 0$ for all $n$.
}
\end{example}

\begin{example} \label{d20}
{\em
$D=-24$: We have $k_1=12+12\sqrt{2}$,
$k_2=60+24\sqrt{2}$, $k_3=84+72\sqrt{2}$ with
$m_1=9+7\sqrt{2}$, $m_2=14+12\sqrt{2}$, $k_0=1$, and
$K=\Q(\sqrt{2})$. Reducing mod $5 \O_K$ yields
$\overline{q}_{n+48}(t)=\overline{q}_{n}(t)$ for all $n \gqs 0$. We have $q_n(0) \not \equiv 0
\bmod 5 \O_K$ for all $n < 48$  and hence $c_{\z_{-24}}(\Delta,n)\neq 0$ for all $n$.
}
\end{example}

Note that in the above examples we have cases with
  class number $h(\Q(\z_{D}))=2$ (namely $D=-15, -20, -24$) and class number
  $h(\Q(\z_{D}))=1$ (the remaining cases).
With Propositions \ref{mod5}, \ref{mod7} and Examples \ref{d7} - \ref{d20} we have completed the proof of Theorem \ref{ell lehmer}.

\section{Arithmetic of Fourier coefficients at CM points}\label{arit}

\subsection{Periodicity} \label{Periodicity}
 As proved in  Theorem \ref{mainthm}, recall our main formula
\begin{equation*}
c_\z(\Delta,m) = \kappa_{\z}\frac{\lambda_{\z}^m}{m!}q_{m, \z}(0) \qquad (m \gqs 0)
\end{equation*}
for the $m$-th Fourier coefficient of $\Delta(z)$ at a CM point $\z \in \H$.
In
this section we examine the integers $q_{m,\z}(0)$ modulo $l \O_K$ in greater detail. For example,
when $l=5$ we have seen in Proposition \ref{mod5} that  $q_{m,\z_{-4}}(0) \bmod 5$ is
\begin{equation}\label{list1}
1, 0, 3, 0, 1, 0, 3, 0, 1, 0, 3, 0, 1, 0, 3, 0, 1, 0, 3, 0,  \cdots
\end{equation}
for $m \gqs 0$ with period $4$. For $l=7$ we have $q_{m,\z_{-4}}(0) \bmod 7$ equaling
\begin{equation}\label{list2}
1, 0, 2, 0, 6, 0, 1, 0, 5, 0, 0, 0, 4, 0, 0, 0, 4, 0, 0, 0, 2, 0, 0,
0, 0, 0, 0, 0, \cdots
\end{equation}
for $m \gqs 0$ with all further terms $\equiv 0 \bmod 7$. Hence $q_{m,\z_{-4}}(0) \bmod 7$ has period $1$ for $m \gqs 21$. We next prove that \eqref{list1}, \eqref{list2} are typical. (In what follows, by {\em period} we always understand the least eventual period of a sequence.)

\begin{theorem} \label{periodic}
Let $l$ be in $\Z_{\gqs 1}$ and $\z=\z_D$ a CM point.
\begin{enumerate}
\item The sequence $q_{m,\z}(0) \bmod l \O_K$ becomes periodic.
\item If $q_{m,\z}(0) \bmod l \O_K$ is periodic from $m=\alpha$ with
 period $\beta$ then $\alpha+\beta \lqs l\abs{\O_K/l \O_K}^{2l}$.
\end{enumerate}
\end{theorem}

\begin{proof}
Recall \eqref{rlz} and the map $\O_K[t] \to\mathcal R_l$ given by $p \mapsto \overline p$. Since $\frac{d}{dt}t^l\equiv 0 \bmod l$ we see that
$\overline{q}_n(t)$ also satisfies the  recursion
(\ref{simplified-writing}) which depends only on $n \bmod l$. See Remark \ref{simplification}. We can
therefore conclude that if, for some rational integers $i_0<j_0$,
\begin{equation} \label{ov}
  \overline{q}_{i_0}(t)=\overline{q}_{j_0}(t), \qquad
  \overline{q}_{i_0+1}(t)=\overline{q}_{j_0+1}(t),\qquad i_0\equiv j_0 \bmod   l,
\end{equation}
then $\overline{q}_{i_0+n}(t)=\overline{q}_{j_0+n}(t)$ for all
$n\in\Z_{\gqs 0}$. But since
$$
(\overline{q}_{i_0},\overline{q}_{i_0+1}, \overline{i_0}), \ \ (\overline{q}_{j_0},\overline{q}_{j_0+1}, \overline{j_0}) \ \in \
\mathcal R_l^2\times (\Z/l\Z),
$$ the box principle implies that
\eqref{ov} is true
for some $0 \lqs i_0<j_0\lqs \abs{\mathcal R_l^2\times
  (\Z/l\Z)}=l\abs{\O_K/l\O_K}^{2l}$. Therefore $q_m(0)\bmod l$ is
periodic from at most $m=i_0$ with period dividing $j_0-i_0$. \end{proof}

\subsection{A simpler recurrence} \label{Periodicity2}
Part of the complication of  recursion \eqref{simplified-writing} is that each polynomial $q_{n+1, \z}(t)$ in the sequence depends on the two before: $q_{n, \z}(t)$ and $q_{n-1, \z}(t)$. Working in $\mathcal R_l$, we see next that each element  in the subsequence
$$
1, \ \overline{q}_{l, \z}(t), \ \overline{q}_{2l, \z}(t), \ \overline{q}_{3l, \z}(t), \ \cdots
$$
depends only on its immediate predecessor.

Let $X= \overline{q}_{nl, \z}(t) \in \mathcal R_l$ for some $n$. Then $\overline{q}_{nl+1, \z}(t)$ is given by  \eqref{simplified-writing}. Since $nl(nl+11)a_4(t)\overline{q}_{nl-1}(t) \equiv 0$ in $\mathcal R_l$ we do not need to know $\overline{q}_{nl-1}(t)$. Continuing inductively, we get
\begin{eqnarray*}
\overline{q}_{nl, \z}(t) &\equiv& X \\
  \overline{q}_{nl+1, \z}(t) &\equiv& a_{2}(t)X+a_{3}(t)X' \\
  \overline{q}_{nl+2, \z}(t) &\equiv& \bigl(a_{1}(t)+ a_{2}(t)\bigr) \bigl(a_{2}(t)X+a_{3}(t)X' \bigr) + {a_{3}}(t)\bigl(a_{2}(t)X+a_{3}(t)X' \bigr)'+ 12 a_{4}(t)X\\
  &\vdots & \\
  \overline{q}_{nl+l, \z}(t) &\equiv& f^{l,0}_\z(t) X + f^{l,1}_\z(t)  a_{3}(t) X'+ \cdots +f^{l,l}_\z(t)  a_{3}(t)^l X^{(l)}
\end{eqnarray*}
with polynomials $f^{l,i}_\z(t) \in \mathcal R_l$ independent of $X$ and $n$.
Let $\Psi_{l,\z}: \mathcal R_l \to \mathcal R_l$ be the linear map
\begin{equation}\label{psi1}
\Psi_{l,\z}(X):= f^{l,0}_\z(t) X + f^{l,1}_\z(t)  a_{3}(t) X'+ \cdots +f^{l,l}_\z(t)  a_{3}(t)^l X^{(l)}.
\end{equation}
We have
\begin{equation}\label{seq}
    \overline{q}_{l, \z}= \Psi_{l,\z}(1), \quad \overline{q}_{2l, \z}= \Psi_{l,\z}(\Psi_{l,\z}(1)), \quad \cdots ,\quad \overline{q}_{ml, \z}= \Psi_{l,\z}^m(1) \text{ \ in \ }\mathcal R_l
\end{equation}
so that the sequence $\overline{q}_{ml, \z}$ for $m=0,1,2, \dots$ satisfies the simple recurrence
$$
\overline{q}_{0, \z}=1, \quad \overline{q}_{(n+1)l, \z} = \Psi_{l,\z}(\overline{q}_{nl, \z}).
$$
The sequence $\Psi_{l,\z}^m(1)$ must eventually repeat for large enough $m$ since
$
\mathcal R_l
$
is finite. It follows that Theorem \ref{periodic}, part (ii) may be improved to:
$$
\text{If $q_{m,\z}(0) \bmod l\O_K$ is periodic from $m=\alpha$ with
   period $\beta$ then $\alpha+\beta \lqs l\abs{\O_K/l\O_K}^{l}$.}
$$

\subsection{Examples and conjectures}
In this part we systematically consider the periodicity of all sequences
\begin{equation}\label{lpq}
\overline{q}_{0, \z_D}(t), \ \overline{q}_{1, \z_D}(t), \ \overline{q}_{2, \z_D}(t), \ \cdots \ \text{ in } \  \mathcal R_l
\end{equation}
for  $D \in \mathscr D$ and $l$ prime  with $2 \lqs l< 100$. For each such pair $(D,l)$ we
use the techniques of \S \S \ref{Periodicity}, \ref{Periodicity2} to
find the least eventual period of \eqref{lpq}. Clearly, the  period of the corresponding sequence of constant terms, $\overline{q}_{m, \z_D}(0)$, will be a divisor. In the simplest case
$$
\overline{q}_{n, \z_D}(t) = \overline{q}_{n+1, \z_D}(t) =0
$$
for some $n$. Then all subsequent terms in the sequence will vanish,
the  period is $1$, and we have $q_{m,\z_D}(0) \to 0 \bmod l \O_K$. This
always happens for $l=2$, $3$, or  $l$ dividing $|D|$, as in these
cases we have $\overline{q}_{1, \z_D}(t) = \overline{q}_{2, \z_D}(t)
=0$. For primes $l$ with $3<l<100$ and  $l$ not dividing $|D|$, our computations show that whether $q_{m,\z_D}(0) \rightarrow 0 \mod l$ depends on the value of $l$ modulo $D$:
If $D \in \{-3,-4,-7,-11, -19\}$  we have
\begin{equation}\label{quadratic residue}
q_{m,\z_D}(0) \not\rightarrow 0 \mod l \iff l \text{ \ \ is a quadratic residue}\mod D
\end{equation}
and if $D \in \{-8,-15, -20, -24\}$  we have
\begin{equation}\label{not quadratic residue}
q_{m,\z_D}(0) \not\rightarrow 0 \mod l \O_K \iff \begin{cases} l \equiv 1,3 & \mod D=-8 \\
l \equiv 1,2,4,8 & \mod D=-15 \\
l \equiv 1,3,7,9 & \mod D=-20 \\
l \equiv 1,5,7,11 & \mod D=-24.
\end{cases}
\end{equation}

\begin{example}
{\em
Let $\z=\z_{-19}$. Modulo $l=41$ we compute that $\overline{q}_{m,\z}(t) = 0$ in $\mathcal R_{41}$ for $m \gqs 1219$ implying $q_{m,\z}(0) \rightarrow 0 \bmod 41$. This agrees with \eqref{quadratic residue} since $41 \equiv 3$ is a non-residue mod $19$.
 Modulo $l=43$ we have $\overline{q}_{m,\z}(t) =  \overline{q}_{m+43\cdot42,\z}(t)$ in $\mathcal R_{43}$ for $m \gqs 32$. We also see that in the first period there exist $\overline{q}_{m,\z}(0) \neq  0$. Therefore $q_{m,\z}(0) \not\rightarrow 0 \bmod 43$. This also agrees with \eqref{quadratic residue} because  $43 \equiv 5$ is a quadratic residue mod $19$.
}
\end{example}

On the above evidence, it is natural to conjecture that  \eqref{quadratic residue} and \eqref{not quadratic residue}  are in fact true for all primes $l>3$. Thus for each $D \in \mathscr D$ we obtain simple criteria for the primes $l$ for which the sequence $q_{m,\z_D}(0) \mod l \O_K$ vanishes for large $m$.
Changing the point of view to consider $D$ modulo $l$, Gautam Chinta gave a succinct reformulation of \eqref{quadratic residue} and \eqref{not quadratic residue} that we state as a Theorem:

\begin{theorem}\label{th2}
Let $D \in \mathscr D$ and $l$ any prime  with $3<l<100$ and  $l$ not dividing $|D|$.
Then
\begin{equation}\label{residue2}
q_{m,\z_D}(0) \not\rightarrow 0 \mod l \O_K \iff D \text{ \ \ is a quadratic residue}\mod l.
\end{equation}
\end{theorem}

Proving that \eqref{quadratic residue} and \eqref{not quadratic residue} are equivalent to Theorem \ref{th2}  is an exercise in quadratic reciprocity.
Note  that \eqref{quadratic residue}, \eqref{not quadratic residue} and \eqref{residue2} are essentially independent of the normalization \eqref{ty}; if $\kappa_\z$ and $\lambda_\z$ are changed by factors in $\O_K$ (so that $q_{m,\z}$ remains in $\O_K[t]$) then \eqref{quadratic residue}, \eqref{not quadratic residue} and \eqref{residue2} will only be affected for finitely many primes $l$. We may further ask if Theorem \ref{th2} is also true for all negative fundamental discriminants $D$.

Next, we look at  the maps $\Psi_{l,\z}$ of \S \ref{Periodicity2} in greater depth. For $l$ prime, they are strikingly simple. In every case  examined, the coefficients $f^{l,i}_\z$ in \eqref{psi1}  are identically zero for $2\lqs i \lqs l-1$. Since $X^{(l)} \equiv 0$ in $\mathcal R_l$ our computations have shown the following.
\begin{prop} \label{px}
 For $D \in \mathscr D$  and all primes $l<100$ there exist $a$, $b \in \mathcal R_l$ such that
$$
\Psi_{l,\z}(X) = a X + b X'.
$$
\end{prop}

\begin{example}
{\em
Let $\z=\z_{-4}$. Then $a_3(t)=6(t^2-1)$ and we have
\begin{eqnarray*}
  \Psi_{5,\z}(X) &=& (2t) X + a_3(t) X' \\
  \Psi_{7,\z}(X) &=& (5t) X + t^2 \cdot a_3(t) X' \\
  \Psi_{11,\z}(X) &=& (7t^3+5t) X + t^2 \cdot a_3(t) X' \\
  \Psi_{13,\z}(X) &=& (5t^3+7t) X + (12 t^4+ 5 t^2 +10) a_3(t) X' .
\end{eqnarray*}
}
\end{example}

\begin{example}
{\em
 In the calculations proving Theorem \ref{th2}, the
 largest  period we found for \eqref{lpq} was $23439864$ when $\z=\z_{-15}$ and $l=83$. To see this, we iterate   $\Psi_{l,\z}$ starting at $1$ and find  $$\Psi_{l,\z}^{83}(1)=\left(11+57\sqrt{5}\right)\Psi_{l,\z}^1(1)
 $$ with no two powers $< 83$ being equal up to a factor.
 The number $11+57\sqrt{5}$ has order $3444 \bmod l \O_K$ and so
 $\Psi_{l,\z}^m(1)$ has  period $82\cdot 3444$. It follows that
 $\overline{q}_{l(82\cdot 3444+1)+m, \z_D}(t)=\overline{q}_{l+m, \z_D}(t)$
 in $\mathcal R_l$  for $m\gqs 0$. Therefore $\overline{q}_{m, \z_D}(t)$ must have
  period $\beta$ dividing  $l \cdot 82 \cdot 3444$. If $l\vert
 \beta$ then $\beta/l$ is the period of $\Psi_{m,\z}(1)$ so $\beta=l\cdot
 82\cdot 3444$.  If $l\not \vert
 \beta$ then $\beta$ is the period of $\Psi_{m,\z}(1)$ so $\beta=82 \cdot 3444$. Checking
 this, we find $\beta$ is $l \cdot 82 \cdot 3444$. The period of the constant term sequence $\overline{q}_{m, \z_D}(0)$ must  be a factor and a final check shows it is $82\cdot 3444$.
 }
\end{example}

Thus we notice  interesting relationships between the  periods of the sequences $\Psi_{l,\z}^m(1)$, $\overline{q}_{m,\z}(t)$ and $\overline{q}_{m,\z}(0)$.
The first, between $\Psi_{l,\z}^m(1)$ and $\overline{q}_{m,\z}(t)$, can be shown in general. Recall the polynomials $a_i(t)$, depending on $\z$, defined in \eqref{simplified-writing}.

\begin{lemma}\label{onemorelemma}
Let $l>2$ be a prime and $\z=\z_D$ any CM point. If $\overline{q}_{m,\z}(t)$ eventually has  period $\beta$, not divisible by $l$, then
\begin{equation}\label{maybethelast}
\overline{a}_2(t) \overline{q}_{n,\z}(t) \to 0\quad \textrm{ and } \quad \overline{a}_4(t) \overline{q}_{n,\z}(t) \to 0 \textrm{ \ in \ } \mathcal
R_l.
\end{equation}
\end{lemma}
\begin{proof}
Suppose $\overline{q}_{n+\beta,\z}(t)=\overline{q}_{n,\z}(t)$  for $n\gqs \alpha$. Let $n$ be any integer $>\alpha$.
Then for every positive $m$  we have
\begin{equation}\label{rio}
\overline{q}_{n+m\beta+1,\z}(t) = \overline{q}_{n+1,\z}(t).
\end{equation}
Applying the recursion \eqref{simplified-writing} to both sides of \eqref{rio} and equating corresponding parts  yields
$$
m \beta\Bigl(\overline{a}_2(t) \overline{q}_{n,\z}(t)+(2n+11+m\beta)\overline{a}_4(t)\overline{q}_{n-1,\z}(t)\Bigr) = 0 \quad \text{in \ } \mathcal R_l.
$$
Since $\beta$ is coprime to $l$ we may choose $m$ so that  $m\beta$
takes any value mod $l$. If this value is non-zero, $m\beta$ is a unit
in $\mathcal R_l$.   Therefore
$$
\overline{a}_2(t) \overline{q}_{n,\z}(t)+(2n+11+v)\overline{a}_4(t)\overline{q}_{n-1,\z}(t) = 0 \quad \text{in \ } \mathcal R_l
$$
for any non-zero $v \bmod l$. Using this with $v=2$ and $v=1$ and
subtracting these two equalities gives
$\overline{a}_4(t)\overline{q}_{n-1}(t)=0$ in $\mathcal R_l$ when $n>\alpha$. From this it
 further follows that $\overline{a}_2(t)\overline{q}_{n,\z}(t)=0$ in $\mathcal R_l$
when $n>\alpha$.
\end{proof}

Label the ideal in $\O_K/l\O_K$ generated by $\overline{a}_2(0)$ and $\overline{a}_4(0)$ as $(\overline{a}_2(0),\overline{a}_4(0))$ .


\begin{prop}\label{pwp} Let $l>2$ be any prime and $\z=\z_D$ any CM point. If
  $(\overline{a}_2(0),\overline{a}_4(0)) = \O_K/l\O_K$ and $\overline{q}_{m,\z}(0) \not\rightarrow 0$ in $\mathcal R_l$ then
\begin{equation}\label{cj1}
 \operatorname{period}\left(\overline{q}_{m,\z}(t)\right) = l \cdot \operatorname{period}\left(\Psi_{l,\z}^m(1)\right).
\end{equation}
\end{prop}
\begin{proof}
We claim that the period of $\overline{q}_{m,\z}(t)$, call it $\beta$,
is divisible by
$l$. Suppose not. Then it follows from Lemma \ref{onemorelemma} that
$\overline{a}_2(t) \overline{q}_{n,\z}(t) \to 0$ and $\overline{a}_4(t) \overline{q}_{n,\z}(t) \to 0$ in $\mathcal
R_l$. But since $(\overline{a}_2(0),\overline{a}_4(0)) = \O_K/l\O_K$, there exist $x, y \in \O_K/l\O_K$ such that $x \overline{a}_2(0)+ y \overline{a}_4(0)= 1$ in $\O_K/l\O_K$.
We can conclude that $\overline{q}_{n,\z}(0) \to 0$. But this contradicts our assumptions. Therefore $l | \beta$. Hence $\Psi_{l,\z}^m(1)$ (which equals $\overline{q}_{ml,\z}(t)$) has period $\beta/l$.
\end{proof}
The  conditions of Proposition \ref{pwp} are usually satisfied. For $D\in \mathscr D$, $l$ prime with $2<l<100$ and $\overline{q}_{m,\z_D}(0) \not\rightarrow 0$, we have $(\overline{a}_2(0),\overline{a}_4(0)) \neq \O_K/l\O_K$ only for the three pairs $(D,l)$ with $D=-11,-19,-24$ and $l=5$.

The next proposition, proved by computation, shows how the periods of $\overline{q}_{m,\z}(t)$ and
$\overline{q}_{m,\z}(0)$ are related.
\begin{prop}\label{qper}
Let $D \in \mathscr D$ and $l$ any prime $<100$. Suppose $(\overline{a}_2(0),\overline{a}_4(0)) = \O_K/l\O_K$ and $\overline{q}_{m,\z_D}(t) \not\rightarrow 0$. Then
\begin{equation}\label{cj}
 \operatorname{period}(\overline{q}_{m,\z_D}(t)) = l \cdot \operatorname{period}(\overline{q}_{m,\z_D}(0)).
\end{equation}
\end{prop}
We expect that Propositions \ref{px} and \ref{qper} hold for all negative fundamental discriminants $D$ and all primes $l$.

\section{Vanishing of $c_z(\Delta,m)$ and $P_{z,m}$ for $z \in \H$} \label{numer}
We have seen, with the proof of Theorem \ref{ell lehmer} in \S \ref{cma}, that $A_\Delta$ does not contain CM points $\z_D$ of small discriminant. To test the inclusion of general points in $\H$ we
define
$$
\e6_m(z):=\sum_{r=0}^m \frac{m!}{r!} \binom{m+11}{r+11} \left(E^*_2\right)^{m-r}\mathcal B_{r}
$$
with $\mathcal B_{r}$ given by \eqref{pn1}.
From Theorem \ref{taylor} we see that
$$c_z(\Delta,m) =0 \iff \e6_m(z)=0$$ and from (\ref{preln})
this happens precisely when $P_{z,m}$ vanishes identically.
The first few $\e6_m$ are
\begin{eqnarray*}
  \e6_0 &=& 1 \\
  \e6_1 &=& 12\left[E_2^*\right] \\
  \e6_2 &=& 12\left[13(E_2^*)^2-E_4\right] \\
  \e6_3 &=& 24\left[91(E_2^*)^3- 21E_2^*E_4+2E_6\right]\\
  \e6_4 &=& 72\left[455 (E_2^*)^4 - 210 (E_2^*)^2 E_4 + 40 E_2^* E_6 +3 E_4^2\right].
\end{eqnarray*}
Let $\F$ be the usual fundamental domain for $\G$, as shown twice in Figure 1. For all $m$, we wish to locate the zeros of $\e6_m$ in $\F$.



\SpecialCoor
\psset{griddots=5,subgriddiv=0,gridlabels=0pt}

\begin{figure}[h]
\begin{center}
\begin{pspicture}(-4.8,-0.5)(12,6.5) 

\psset{linewidth=1pt}

\pscustom[fillstyle=gradient,linecolor=white,gradmidpoint=1,gradbegin=byellow,gradend=lightblu2,gradlines=50]{%
  \psline[linewidth=1pt](-1,6.5)(-1,1.732)
  \psarcn[linewidth=1pt](0,0){2}{120}{90}
  \psarcn[linestyle=dashed,linecolor=gray](0,0){2}{90}{60}
  \psline[linestyle=dashed,linecolor=gray](1,1.732)(1,6.5)
  \psline[linecolor=white](1,6.5)(-1,6.5)
  }

\psline[linewidth=1pt](-1,6.5)(-1,1.732)
  \psarcn[linewidth=1pt](0,0){2}{120}{90}
  \psarcn[linestyle=dashed,linecolor=gray](0,0){2}{90}{60}
  \psline[linestyle=dashed,linecolor=gray](1,1.732)(1,6.5)

\psline[linecolor=gray](0,0)(0,0.15)
\psline[linecolor=gray](-1,0)(-1,0.15)
\psline[linecolor=gray](1,0)(1,0.15)

\rput(-3.3,2){$1$}
\rput(-3.3,4){$2$}
\rput(-3.3,6){$3$}

  \psline[linewidth=1pt,linecolor=gray](-3.5,0)(3,0)

\psline[linecolor=gray](-3,2)(-2.85,2)
\psline[linecolor=gray](-3,4)(-2.85,4)
\psline[linecolor=gray](-3,6)(-2.85,6)

  \psline[linewidth=1pt,linecolor=gray](-3,-0.5)(-3,6.5)

  \rput(0,-0.3){$0$}
  \rput(1,-0.3){$1/2$}
  \rput(-1,-0.3){$-1/2$}
  \rput(2.3,4){zeros of $\e6_6$}
  \rput(-1.5,5){$\F$}

  \rput(2.5,1){$\H$}

\psdot*(0,2.32)
\psdot*(-0.579,1.914)
\psdot*(0,3)
\psdot*(0,3.7)
\psdot*(0,5.2)
\psdot*(-1,2.14)
\psdot*(-1,2.76)
\psdot*(-1,3.9)
\psdot*(-1,5.2)


 \psline[linewidth=1pt,linecolor=gray](4.5,0)(11,0)
 \psline[linewidth=1pt,linecolor=gray](5,-0.5)(5,6.5)

\pscustom[fillstyle=gradient,linecolor=white,gradmidpoint=1,gradbegin=byellow,gradend=lightblu2,gradlines=50]{%
  \psline[linewidth=1pt](7,6.5)(7,1.732)
  \psarcn[linewidth=1pt](8,0){2}{120}{90}
  \psarcn[linestyle=dashed,linecolor=gray](8,0){2}{90}{60}
  \psline[linestyle=dashed,linecolor=gray](9,1.732)(9,6.5)
  \psline[linecolor=white](9,6.5)(7,6.5)
  }

\psline[linewidth=1pt](7,6.5)(7,1.732)
  \psarcn[linewidth=1pt](8,0){2}{120}{90}
  \psarcn[linestyle=dashed,linecolor=gray](8,0){2}{90}{60}
  \psline[linestyle=dashed,linecolor=gray](9,1.732)(9,6.5)

\psline[linecolor=gray](8,0)(8,0.15)
\psline[linecolor=gray](7,0)(7,0.15)
\psline[linecolor=gray](9,0)(9,0.15)

\rput(4.7,2){$1$}
\rput(4.7,4){$2$}
\rput(4.7,6){$3$}

\psline[linecolor=gray](5,2)(5.15,2)
\psline[linecolor=gray](5,4)(5.15,4)
\psline[linecolor=gray](5,6)(5.15,6)

  \rput(8,-0.3){$0$}
  \rput(9,-0.3){$1/2$}
  \rput(7,-0.3){$-1/2$}
  \rput(10.3,4){zeros of $\e6_8$}
  \rput(6.5,5){$\F$}

\rput(10.5,1){$\H$}

\psdot*(8,2.32)
\psdot*(7.598,1.959)
\psdot*(8,3.1)
\psdot*(8,4.04)
\psdot*(8,4.9)
\psdot*(8,6.4)
\psdot*(7,1.732)
\psdot*(7,2.02)
\psdot*(7,3.1)
\psdot*(7,3.69)
\psdot*(7,5)
\psdot*(7,6.4)

\psdot*(7.282, 2.928)
\psdot*(7.486, 2.456)
\psdot*(7.806, 2.32)
\psdot*(8.718, 2.928)
\psdot*(8.514, 2.456)
\psdot*(8.194, 2.32)

\end{pspicture}
\caption{The zeros of $\e6_6$ and $\e6_8$\label{pfig}}
\end{center}
\end{figure}


\begin{lemma} \label{l8}
We have $\lim_{y \to \ci} \e6_m(z)=12^m$, uniformly in $x$.
\end{lemma}
\begin{proof}
Combine \eqref{czm}, Theorem \ref{taylor} and the definition of $\e6_m$ to see that
\begin{equation}\label{dem}
\e6_m=12^m \partial^m \Delta/ \Delta
\end{equation}
so that $\e6_m$ is a non-holomorphic modular form of weight $2m$ satisfying the recursion $\e6_{m+1}=12 \partial \e6_m+\e6_1 \e6_m$. Note also that
\begin{equation}\label{dem2}
    \lim_{y \to \ci} \Delta^{(r)}(z)/\Delta(z) = (2\pi i)^r.
\end{equation}
The lemma now follows from \eqref{dmfz}, \eqref{dem} and \eqref{dem2}.
\end{proof}

\begin{lemma} For each $m$, the zeros of $\e6_m$
\begin{enumerate}
\item form a set of measure $0$,
\item are contained in a bounded region of $\F$,
\item are symmetric about the line $\Re(z)=0$.
\end{enumerate}
\end{lemma}
\begin{proof}
As a consequence of Lemma \ref{l8},  $\e6_m \not \equiv 0$. Therefore, since $\e6_m$ is real analytic, we obtain (i).
Part (ii) also follows from Lemma \ref{l8}.
 The Fourier coefficients at $\ci$ of $E_2^*$, $E_4$ and $E_6$ are real, implying that
 \begin{equation}\label{chi}
 \e6_m(-\overline{z})=\overline{\e6_m(z)}
\end{equation}
and hence (iii).
\end{proof}

\begin{lemma} \label{lemh}
We have $\e6_m(z) \in \R$ on the vertical lines $\Re(z)=-1/2$ and $\Re(z)=0$. We have $e^{im\theta} \e6_m(e^{i\theta}) \in \R$ for $0<\theta<\pi$.
\end{lemma}
\begin{proof}
For $\Re(z)=-1/2$ and $\Re(z)=0$ we have $\e6_m(z) \in \R$ using (\ref{chi}) and that $\e6_m(z+1)=\e6_m(z)$. As noted in \cite{RankinSwinnerton-Dyer:1970a} we have $e^{i k \theta/2}E_k(e^{i\theta}) \in \R$ for $0<\theta<\pi$ and $k>2$. Employing Hecke's limiting argument, as in \cite[p. 19]{BruinierGeerHarderZagierRanestad:2008a} for example, we also find that $e^{i  \theta}E^*_2(e^{i\theta}) \in \R$. As $\e6_m$ is a polynomial in $E_2^*$, $E_4$ and $E_6$, homogeneous in the weight, we obtain $e^{im\theta} \e6_m(e^{i\theta}) \in \R$ as required.
\end{proof}

Thus we may expect zeros for $\e6_m$ along the boundary of the left
half of $\F$ as the real-valued functions in Lemma \ref{lemh} change
sign. At the corners of this boundary we have $\e6_m(i)=0$ for $m
\equiv 1 \bmod 2$ and $\e6_m(\omega)=0$ for $m \equiv 1,2 \bmod 3$ by
Remark \ref{trivial-vanishing}.
In what follows we locate numerically the zeros of $\e6_m$ in  $\F$, using the first 25 terms in the Fourier expansions at $\ci$ of $E_2^*$, $E_4$ and $E_6$  to get approximations to $\e6_m$.

\vskip 3mm
\noindent
{\bf Zeros of $\e6_1$.} We have $\e6_1(z)=0$  for $z=i, \omega$. Numerically there seem to be no further zeros.

\vskip 3mm
\noindent
{\bf Zeros of $\e6_2$.} We have $\e6_2(z)=0$  for $z= \omega$ and  the points
$$
1.344 i, \quad -1/2+1.29i.
$$
\begin{proof}[Proof of Theorem \ref{0}.]
The  value $1.344 i$ lies between the CM points $i=\z_{-4}$ and $\sqrt{2} i=\z_{-8}$, where we may explicitly compute $\e6_2$. With  the table on p. 87 of \cite{BruinierGeerHarderZagierRanestad:2008a} we compute $\e6_2(\z_{-4})=-144 \Omega_{-4}^4 <0$ and $\e6_2(\z_{-8})=72 \Omega_{-8}^4 >0$. Therefore the real-valued $\e6_2(z)$  changes sign between $i$ and $\sqrt{2} i$, proving that there is a zero in the vicinity of $1.344 i$.
\end{proof}

\vskip 3mm
\noindent
{\bf Zeros of $\e6_3$.} We have $\e6_3(z)=0$  for $z= i$ and  the points
$$
1.666i, \quad -1/2+1.642 i, \quad -1/2+1.155i.
$$

\vskip 3mm
\noindent
{\bf Zeros of $\e6_6$.} The 9 inequivalent zeros of $\e6_6(z)$  are
shown on the left of Figure 1. They all have real part $0$ or $-1/2$ or absolute value $1$.

\vskip 3mm
\noindent
{\bf Zeros of $\e6_7$.}  This is the smallest $m$ for which $\e6_m(z)$ appears to have a pair of zeros not on the boundary of the left half of $\F$. The pair is $\pm 0.302 +1.18i$.
The remaining 11 zeros of $\e6_7(z)$ are on the boundary.

\vskip 3mm
\noindent
{\bf Zeros of $\e6_8$.} The 18 zeros of $\e6_8(z)$  are
shown on the right of Figure 1. Three pairs do not have real part $0$ or $-1/2$ or absolute value $1$.

\section{Analogues of Petersson's formula} \label{analogs}\label{more on poincareseries}
We return to the case of a general, discrete, finitely generated, finite volume group $\G \subseteq \SL(2,\R)$, and assume for simplicity that $\G$ has a cusp of width 1 at $\infty$.
Let $\mathcal{F}$ be an orthonormal basis for
  $S_k(\Gamma)$. Petersson first proved the following
formula:
\begin{equation}\label{petform}
    \sum_{f\in
  \mathcal{F}} c_\infty(f,n) \overline{c_\infty(f,m)}  = \frac{(4\pi \sqrt{mn})^{k-1}}{(k-2)!}\left( \delta_{mn} + 2\pi (i^{-k}) \sum_{c >o} c^{-1} S(m,n;c)J_{k-1}\left(\frac{4\pi \sqrt{mn}}{c}\right)\right).
\end{equation}
Here the sum on the right runs over positive lower right entries of the group $\G$, $S(m,n;c)$ is the Kloosterman sum related to $\G$, and $J$ is a
Bessel function - see for example
\cite[Prop. 14.5]{IwaniecKowalski:2004a}  or \cite[Theorem
9.6]{Iwaniec:2002a} for further details.
In this section we find explicit expressions for
\begin{equation}\label{newsums}
\sum_{f\in \mathcal{F}}c_{z_0}(f,m) \overline{c_\infty(f,n)}  \quad
\text{ and } \quad \sum_{f\in \mathcal{F}} c_{z_0}(f,m) \overline{c_{z_0'}(f,n)}
\end{equation}
with $z_0$, $z_0' \in \H$.
Since the non-cuspidal Fourier coefficients $c_{z_0}(f,m)$ are essentially non-holomorphic  weight $k+2m$ forms
(recall (\ref{czm})), we can relate the two sums in (\ref{newsums}) to special
values of  non-holomorphic objects of weight $k+2m$ in $z_0$ and weight $k+2n$ in $z_0'$, which we define below.

\subsection{Averages of products of cuspidal and non-cuspidal Fourier coefficients}
Recall our notation $\beta=\Im(z_0)$.
Consider the Fourier expansion  at $z_0$ of the  Poincar\'e
series associated to $\ci$ and the Fourier expansion at infinity of the Poincar\'e series associated to $z_0$:
\begin{eqnarray*}
  P_n|_k \sz (z) &=& \sum_{l \in \Z_{\geqslant 0}} c_{z_0}(P_n, l)z^l, \\
  P_{z_0,m}(z) &=& \sum_{l \in \Z_{\geqslant 1}} c_\infty(P_{z_0,m},l) q^l.
\end{eqnarray*}
Computing $\s{P_n}{P_{z_0,m}}$ with Propositions \ref{porth} and \ref{eorth} and equating the results we obtain
$$
\overline{c_\infty(P_{z_0,m},n)} = \frac{\pi (4 \pi n)^{k-1}}{2^{k-3} (m+k-1)!} c_{z_0}(P_n, m).
$$
Using (\ref{czm}) to rewrite the right-hand side gives
\begin{equation}\label{ppp}
\overline{c_\infty(P_{z_0,m},n)} = \frac{\pi (4 \pi n)^{k-1}(-4\pi \beta)^m (2i \beta)^{k/2}}{2^{k-3} (m+k-1)!} \partial^m P_n(z_0).
\end{equation}
To find $\partial^m P_n(z_0)$ more explicitly set
\begin{equation}\label{hk}
F_k(z,n,s):=\sum_{\g \in \Gamma_\infty\backslash\Gamma} \Im(\g z)^{s-k/2} j(\g,z)^{-k} e^{2\pi i n \g z}.
\end{equation}
For $\Re(s)>1$, the series defining $F_k(z,n,s)$ converges absolutely and uniformly to a non-holomorphic, weight $k$ Poincar\'e series. Verify that
$$
\partial_k F_k(z,n,s) = n F_{k+2}(z,n,s+1) -\frac{s+k/2}{4\pi} F_{k+2}(z,n,s)
$$
and by induction
\begin{equation*}
\partial^m F_k(z,n,s) = (-4\pi )^{-m} \sum_{j=0}^m (-4\pi n)^j \binom{m}{j} \frac{\G(s+k/2+m)}{\G(s+k/2+j)} F_{k+2m}(z,n,s+j).
\end{equation*}
Since $F_k(z,n,k/2)=P_n$ we obtain
\begin{equation}\label{pppp}
\partial^m P_n = (-4\pi )^{-m} \sum_{j=0}^m (-4\pi n)^j \binom{m}{j} \frac{(m+k-1)!}{(j+k-1)!} F_{k+2m}(z,n,k/2+j)
\end{equation}
 and combining (\ref{ppp}) with (\ref{pppp}) proves
\begin{prop} \label{ellparexpansion} With the above notation
$$
\overline{c_\infty(P_{z_0,m},n)} = \frac{2 \beta^{m+k/2}}{n(-2i)^{k/2}} \sum_{j=0}^m  \binom{m}{j} \frac{(-4\pi n)^{k+j}}{(j+k-1)!} F_{k+2m}(z_0,n,k/2+j).
$$
\end{prop}

\begin{theorem}\label{petterson} For any $m \in \Z_{\geqslant 0}$, $n \in \Z_{\geqslant 1}$ we have
 $$
\sum_{f\in \mathcal{F}} c_{z_0}(f,m) \overline{c_\infty(f,n)} =
\frac{-(2i)^{k/2}\beta^{m+k/2}}{(k-2)!} \sum_{j=0}^m  \binom{m+k-1}{m-j} \frac{(-4\pi n)^{k+j-1}}{j!}  F_{k+2m}(z_0,n,k/2+j).
$$
\end{theorem}
\begin{proof}
We expand from the basis $\mathcal{F}$ and use Propositions \ref{porth}, \ref{eorth} to find
\begin{align}
\s{P_n}{P_{z_0, m}}&=\sum_{f\in
  \mathcal{F}}\s{P_n}{f}\s{f}{P_{z_0, m}} \label{nsu}\\
& =\sum_{f\in
  \mathcal{F}} \overline{c_\infty(f,n)} \frac{(k-2)!}{(4\pi n)^{k-1}}\frac{\pi (k-2)! m!}{2^{k-3} (m+k-1)!} c_{z_0}(f,m) \nonumber
\end{align}
On the other hand
the left  side of \eqref{nsu} equals
$(k-2)! (4\pi n)^{1-k}\overline{c_\infty(P_{z_0,m},n)}$
and using the expression from Proposition \ref{ellparexpansion}  completes the proof.
\end{proof}

\subsection{Averages of products of non-cuspidal Fourier coefficients}
To get a similar result for the second sum in (\ref{newsums}) we put $$
Q_{m,l}(z):=z^m \left(\frac{\overline{z}}{|z|^2-1} \right)^l
$$
and
\begin{eqnarray} \nonumber
G_k(z,z_0;m,l) & := & \sum_{\g \in \G} \left( Q_{m,l} |_k \sz^{-1} \g \right)(z)\\
& = & (2i\beta)^{k/2} \sum_{\g \in \G} \frac{\left(\sz^{-1} \g z \right)^{m-l}}{j\left(\sz^{-1} \g, z \right)^{k}}
\left( \frac{\left| \sz^{-1} \g z \right|^2}{\left| \sz^{-1} \g z \right|^2-1}\right)^l. \label{gk}
\end{eqnarray}
Then $G_k(z,z_0;m,0) = P_{z_0,m}(z)$ and for all $m \gqs 0$ and $0 \lqs l < k/2-2$ we have (using the methods of \cite[\S \S 4, 5]{ImamogluOSullivan:2009}) that $G_k(z,z_0;m,l)$ converges absolutely and uniformly to a non-holomorphic weight $k$ modular form. Note also that
we may use
$\sigma_{z}0=z$  to express \eqref{gk} in the different form:
\begin{equation}\label{sum2}
G_k(z,z_0;m,l)=\frac{\left(z_0-\overline{z_0}\right)^{k/2}}{\left(z-\overline{z}\right)^{k}} \sum_{\left(\smallmatrix a
& b \\ c & d \endsmallmatrix\right)\in\sz^{-1}\G \sigma_{z}} d^{-k} \left(\frac bd \right)^{m-l}
\left( \frac{\left| \frac bd \right|^2}{\left| \frac bd \right|^2-1}\right)^l.
\end{equation}

\begin{theorem}\label{elpeter} For  all $m,n  \in \Z_{\geqslant 0}$ we have
 \begin{multline}
\sum_{f\in \mathcal{F}}\overline{c_{z_0}(f,m)} c_{z_0'}(f,n)=\frac{2^{k-3} (m+k-1)!\left(z_0'-\overline{z_0'}\right)^{n+k/2}}{\pi m! (k-2)!}\\
\times \sum_{j=0}^n \binom{m}{j} \binom{n+k-1}{n-j} G_{k+2n}(z_0',z_0;m-j,n-j).
\end{multline}
\end{theorem}
\begin{proof}
As in Proposition \ref{petterson},
\begin{align*}
\s{P_{z_0, m}}{P_{z_0', n}}&=\sum_{f\in
  \mathcal{F}}\s{P_{z_0, m}}{f}\s{f}{P_{z_0', n}}\\
& =\sum_{f\in
  \mathcal{F}} \overline{c_{z_0}(f,m)} \frac{\pi (k-2)! m!}{2^{k-3} (m+k-1)!} c_{z_0'}(f,n) \frac{\pi (k-2)! n!}{2^{k-3} (n+k-1)!}
\end{align*}
where the left-hand side is
\begin{equation*}
\frac{\pi (k-2)! n!}{2^{k-3} (n+k-1)!} c_{z_0'}(P_{z_0,m},n).
\end{equation*}
Thus, with (\ref{czm}), it remains to compute $\partial^n P_{z_0,m}(z)|_{z=z_0'}$. Let $\tau$ be any element of $\GL(2,\C)$. Then for any differentiable function $h$ we have $\frac{d}{dz} h(\tau z)= \det \tau \cdot h'(\tau z) j(\tau, z)^{-2}$. Also
$$
\frac{d}{dz} j(\tau, z)^{-k} = \frac{-k}{2iy} j(\tau, z)^{-k} + \frac{k}{2iy} j(\tau, z)^{-k-2} j(\tau,z)j(\tau, \overline{z}).
$$
Hence
\begin{equation} \label{partt}
\partial_k (h|_k \tau) = \frac{1}{2\pi i} \left(h'|_{k+2} \tau +k\left[ \frac{j(\tau,z)j(\tau, \overline{z})}{2iy \det \tau}\right] h|_{k+2} \tau \right).
\end{equation}
Now $j(\tau,z)j(\tau, \overline{z})/(2iy \det \tau) = 1/(\tau z -\tau \overline{z})$. In the case where $\tau = \sz^{-1} \g$ with $\g \in \SL(2,\R)$ we have $\tau \overline{z} = 1/ \overline{\tau z}$ and then it follows from (\ref{partt}) that
$$
\partial_k (Q_{l,m}|_k \tau) = \frac{1}{2\pi i} \left((k-l)Q_{l+1,m}|_{k+2} \tau +m Q_{l,m-1}|_{k+2} \tau \right).
$$
Therefore
$$
\partial_k G_k(z,z_0;m,l) = \frac{1}{2\pi i} \left((k-l)G_{k+2}(z,z_0;m,l+1)  +m G_{k+2}(z,z_0;m-1,l) \right)
$$
and by induction
$$
\partial^n G_k(z,z_0;m,l) = \frac{n!}{(2\pi i)^n} \sum_{j=0}^n \binom{m}{j} \binom{n-l+k-1}{n-j} G_{k+2n}(z,z_0;m-j,l+n-j).
$$
With $l=0$ we find
\begin{equation}\label{partiale}
\partial^n P_{z_0,m} = \frac{n!}{(2\pi i)^n} \sum_{j=0}^n \binom{m}{j} \binom{n+k-1}{n-j} G_{k+2n}(z,z_0;m-j,n-j)
\end{equation}
and the proof is complete.
\end{proof}

\subsection{A second proof of Proposition \ref{ellparexpansion}}
We give another demonstration of the key  Proposition \ref{ellparexpansion}. In this argument the Poincar\'e series $F_k(z,n,s)$ emerges very naturally.

\begin{lemma}\label{ellparexpansion2} We have
  \begin{equation} \label{eq:1}
    c_\ci(P_{z_0,m},n)=
    \frac{2}{n(2i)^{k/2}} \sum_{j=0}^m \beta^{k/2+j} \binom{m}{j} \frac{(-4\pi n)^{k+j}}{(k+j-1)!} \sum_{\genfrac{}{}{0pt}{0}{\gamma\in \Gamma\slash\Gamma_\infty}{\left(\smallmatrix a
& b \\ c & d \endsmallmatrix\right)=\sz^{-1}\gamma}} \frac{a^{m-j}}{c^{k+m+j}} e^{2\pi i nd/c}.
  \end{equation}
\end{lemma}
\begin{proof}
Write
\begin{equation*}
  P_{z_0,m}(z) = 2(2i\beta)^{k/2}\sum_{\gamma\in \Gamma\slash\Gamma_\infty}\sum_{n=-\infty}^\infty \frac{(\sz^{-1} \g (z+n))^{m}}
    {j(\sz^{-1} \g,(z+n))^{k}}.
\end{equation*}
Using Poisson summation on the inner sum we obtain
\begin{equation*}
P_{z_0,m}(z) =  2(2i\beta)^{k/2}\sum_{\gamma\in \Gamma\slash\Gamma_\infty}\sum_{n=-\infty}^\infty I(\gamma,z,n)
  \end{equation*}
where, for $\left(\smallmatrix a
& b \\ c & d \endsmallmatrix\right)=\sz^{-1}\gamma$,
\begin{align*}
I(\gamma,z,n)&=\int_\R\frac{\left(\frac ac -\frac{2 i \beta}{c(c(z+t)+d)}\right)^m}{(c(z+t)+d)^k}e^{-2\pi i nt}\,dt\\
&=\sum_{j=0}^m\binom{m}{j}\left(\frac{a}{c}\right)^{m-j}\left(\frac{-2 i \beta}{c}\right)^{j}\int_{-\infty}^{\infty}\frac{e^{-2\pi i nt}}{(c(z+t)+d)^{k+j}}\,dt\\
&=\sum_{j=0}^m\binom{m}{j}\left(\frac{a}{c}\right)^{m-j}\left(\frac{-2 i \beta}{c}\right)^{j}
e^{2\pi i nd/c}e^{2\pi i nz} \int_{-\infty+iy}^{\infty+iy}\frac{e^{-2\pi i nu}}{(cu)^{k+j}}\,du.
\end{align*}
We note that $c$ is never
zero when $\gamma\in \hbox{SL}_2(\mathbb{R})$, so division by $c$ is
not a problem in the above calculations.
When $n\leqslant 0$ we see, by moving the line of integration upwards,
that the integral vanishes. When $n$ is positive we can move the line
of integration downwards and see that the integral equals
\begin{equation*}
  -2\pi i \operatornamewithlimits{Res}_{u=0}\frac{e^{-2\pi i nu}}{(cu)^{k+m-j}}=\frac{(-2\pi i n/c)^{k+j}}{n (k+j-1)!}
\end{equation*}
Putting together the different terms we obtain the lemma.
\end{proof}

Rewriting the series on the right of (\ref{eq:1}) using the following identities
\begin{eqnarray*}
    \sum_{\genfrac{}{}{0pt}{0}{\gamma\in \Gamma\slash\Gamma_\infty}{\left(\smallmatrix a
& b \\ c & d \endsmallmatrix\right)=\sz^{-1}\gamma}} f(a,b,c,d)
      & = &  \sum_{\genfrac{}{}{0pt}{0}{\gamma^{-1}\in \Gamma\slash\Gamma_\infty}{\left(\smallmatrix a
& b \\ c & d \endsmallmatrix\right)=\sz^{-1}\gamma^{-1}}} f(a,b,c,d)\\
& = &  \sum_{\genfrac{}{}{0pt}{0}{\gamma\in \Gamma_\infty\backslash\Gamma}{\left(\smallmatrix a
& b \\ c & d \endsmallmatrix\right)^{-1}=\g \sz}} f(a,b,c,d)\\
& = &  \sum_{\displaystyle \left(\smallmatrix a'
& b' \\ c' & d' \endsmallmatrix\right) \in \Gamma_\infty\backslash\Gamma} f(c'z_0+d', -a'z_0-b', c' \overline{z_0} +d', -a' \overline{z_0} -b')
\end{eqnarray*}
we conclude that
\begin{eqnarray*}
\sum_{\genfrac{}{}{0pt}{0}{\gamma\in \Gamma\slash\Gamma_\infty}{\left(\smallmatrix a
& b \\ c & d \endsmallmatrix\right)=\sz^{-1}\gamma}} \frac{a^{m-j}}{c^{k+m+j}} e^{2\pi i nd/c}
 & = &
\sum_{\g \in \G_\infty\backslash\G}\frac{j(\g,z_0)^{m-j}}{j(\g,\overline{z_0})^{k+m+j}}e^{-2\pi i n \g \overline{z_0}}\\
& = &
\beta^{m-j}\sum_{\g \in \G_\infty\backslash\G} \Im(\g z_0)^{j-m} j(\g,\overline{z_0})^{-k-2m} e^{-2\pi i n \g \overline{z_0}}\\
& = &
\beta^{m-j} \overline{F_{k+2m}}(z_0,n,k/2+j).
\end{eqnarray*}
Therefore
\begin{equation*}
  c_\ci(P_{z_0,m},n)=
    \frac{2\beta^{k/2+m}}{n(2i)^{k/2}} \sum_{j=0}^m \binom{m}{j} \frac{(-4\pi n)^{k+j}}{(k+j-1)!} \overline{F_{k+2m}}(z_0,n,k/2+j)
\end{equation*}
and we have arrived at Proposition \ref{ellparexpansion}  by another route.

\subsection{Examples}
We illustrate the results of this section with some particular cases.
\begin{example}
{\em
When $k=12$ and $\Gamma=\SL(2,\Z)$, $S_k(\Gamma)$ is 1-dimensional and $\mathcal{F}=\{\Delta/\norm{\Delta}\}$. In this case Theorem
\ref{petterson} implies that
\begin{equation}\label{ppe}
    \tau(n)c_{z_0}(\Delta,m)= \frac{64}{10!} \| \Delta \|^2 \Im(z_0)^{m+6} \sum_{j=0}^m  \binom{m+11}{m-j} \frac{(-4\pi n)^{j+11}}{j!}  F_{2m+12}(z_0,n,j+6).
\end{equation}
We may verify \eqref{ppe} in the case of $m=0$:
\begin{equation}\label{ppe2}
\tau(n)c_{z_0}(\Delta,0)= \frac{64}{10!} \| \Delta \|^2 \Im(z_0)^{6} (-4\pi n)^{11}  F_{12}(z_0,n,6)
\end{equation}
where $F_{12}(z_0,n,6) = P_n(z_0)$. With \eqref{ptau}, equation \eqref{ppe2} reduces to $c_{z_0}(\Delta,0) = \left(z_0-\overline{z_0}\right)^6 \Delta(z_0)$ which now follows from \eqref{cz0}.
}
\end{example}

\begin{example}
{\em
Consider \eqref{ppe} for $z_0=\z_D$ and $D\in
\mathscr D$. Divide both sides by the  coefficient
$c_{\z_D}(\Delta,m)$, which is non-zero  by Theorem \ref{ell lehmer} (assuming $2|m$ if $\z=\z_{-4}$ or $3|m$ if $\z=\z_{-3}$), and we find
 Lehmer's conjecture is equivalent to showing that for every $n$ there exists an  $m$ and a $\z_D$ as above with
$$
\sum_{j=0}^m  \binom{m+11}{m-j} \frac{(-4\pi n)^{j}}{j!}  F_{2m+12}(\z_D,n,j+6) \neq 0.
$$
}
\end{example}

\begin{example}
{\em
Specializing Theorem \ref{elpeter} to $\G=\SL(2,\Z)$, $k=12$, $m=n$, $z_0=z_0' \in \H$, and employing \eqref{sum2}, we obtain a final vanishing criterion for $P_{z_0,m}$ and $c_{z_0}(\Delta,m)$:
\begin{eqnarray*}
  P_{z_0,m} \equiv 0 & \iff & |c_{z_0}(\Delta,m)|^2 =0 \\
   & \iff & \sum_{j=0}^m \binom{m}{j} \binom{m+11}{j} G_{2m+12}(z_0,z_0;j,j) =0\\
   & \iff & \sum_{j=0}^m \binom{m}{j} \binom{m+11}{j} \sum_{\left(\smallmatrix a
& b \\ c & d \endsmallmatrix\right)\in\sz^{-1}\G \sz} d^{-2m-12}
\left( \frac{\left| \frac bd \right|^2}{\left| \frac bd \right|^2-1}\right)^j =0.
\end{eqnarray*}
}
\end{example}

\bibliography{bibliography}

\end{document}